\documentclass[hidelinks,onefignum,onetabnum]{siamart220329}

\usepackage{color}
\definecolor{mygray}{gray}{.75}
\usepackage{xcolor}
\usepackage{paralist}
\usepackage{hyperref}
\usepackage{amssymb}
\usepackage{tikz,pgfplots}
\usepackage{bm}

\usepackage[textsize=tiny]{todonotes}

\usepackage{pgfplotstable}
\usepackage{cite}

\pgfplotsset{compat=newest,compat/show suggested version=false}
\usepackage{cutwin}

\usepackage{dsfont}

\usepackage{algorithm}
\usepackage{algpseudocode}

\usepackage{caption}
\usepackage[font=footnotesize, width=\textwidth]{caption}

\usepackage{subfigure}
\usepackage{wrapfig}

\usepackage{soul}

\usepackage{xspace}

\makeatletter
\@namedef{subjclassname@2020}{%
  \textup{2020} Mathematics Subject Classification}
\makeatother

\newcommand\ts[1]{{\textstyle{#1}}}

\theoremstyle{plain}
\newtheorem{remark}[theorem]{Remark}
\newtheorem{assumption}[theorem]{Assumption}

\newcommand{\cN}{\mathcal{N}}
\newcommand{\cO}{\mathcal{O}}

\newcommand{\cT}{\mathcal{T}}
\newcommand{\cV}{\mathcal{V}}

\newcommand{\bb}{\bm b} 
\newcommand{\bc}{\bm c} 
 
\newcommand{\bg}{\bm g}

\newcommand{\bu}{\bm u}

\newcommand{\bx}{\bm x} 
\newcommand{\by}{\bm y}

\newcommand{\bA}{\bm A}

\newcommand{\bM}{\bm M} 
\newcommand{\bN}{\bm N}

\newcommand{\C}{\mathbb{C}}
\renewcommand{\H}{\mathbb{H}}
\newcommand{\N}{\mathbb{N}}
\newcommand{\R}{\mathbb{R}}
\newcommand{\U}{\mathbb{U}}
\newcommand{\V}{\mathbb{V}}

\newcommand{\bbA}{\pmb{\mathbb{A}}}
\newcommand{\bbB}{\pmb{\mathbb{B}}}

\newcommand{\bbS}{\pmb{\mathbb{S}}}
\newcommand{\bbU}{\pmb{\mathbb{U}}}
\newcommand{\bbV}{\pmb{\mathbb{V}}}
\newcommand{\bbW}{\pmb{\mathbb{W}}}

\newcommand{\Dt}{{\Delta t}}

\newcommand{\MATLAB}{\textsc{Matlab}\xspace}

\ifpdf
\hypersetup{
  pdftitle={Space-Time Form of the Schrödinger Equation},
  pdfauthor={S. Hain and K. Urban}
}
\fi

\makeatletter
\newcommand{\opnorm}{\@ifstar\@opnorms\@opnorm}
\newcommand{\@opnorms}[1]{%
  \left|\mkern-1.5mu\left|\mkern-1.5mu\left|
   #1
  \right|\mkern-1.5mu\right|\mkern-1.5mu\right|
}
\newcommand{\@opnorm}[2][]{%
  \mathopen{#1|\mkern-1.5mu#1|\mkern-1.5mu#1|}
  #2
  \mathclose{#1|\mkern-1.5mu#1|\mkern-1.5mu#1|}
}
\makeatother
\makeatletter
\DeclareFontFamily{OMX}{MnSymbolE}{}
\DeclareSymbolFont{MnLargeSymbols}{OMX}{MnSymbolE}{m}{n}
\SetSymbolFont{MnLargeSymbols}{bold}{OMX}{MnSymbolE}{b}{n}
\DeclareFontShape{OMX}{MnSymbolE}{m}{n}{
    <-6>  MnSymbolE5
   <6-7>  MnSymbolE6
   <7-8>  MnSymbolE7
   <8-9>  MnSymbolE8
   <9-10> MnSymbolE9
  <10-12> MnSymbolE10
  <12->   MnSymbolE12
}{}
\DeclareFontShape{OMX}{MnSymbolE}{b}{n}{
    <-6>  MnSymbolE-Bold5
   <6-7>  MnSymbolE-Bold6
   <7-8>  MnSymbolE-Bold7
   <8-9>  MnSymbolE-Bold8
   <9-10> MnSymbolE-Bold9
  <10-12> MnSymbolE-Bold10
  <12->   MnSymbolE-Bold12
}{}

\let\llangle\@undefined
\let\rrangle\@undefined
\DeclareMathDelimiter{\llangle}{\mathopen}%
                     {MnLargeSymbols}{'164}{MnLargeSymbols}{'164}
\DeclareMathDelimiter{\rrangle}{\mathclose}%
                     {MnLargeSymbols}{'171}{MnLargeSymbols}{'171}
\makeatother

\usepackage{csquotes}

\DeclareMathOperator*{\Span}{span}

\numberwithin{equation}{section}
\allowdisplaybreaks

\pgfplotsset{select coords between index/.style 2 args={
    x filter/.code={
        \ifnum\coordindex<#1\fi
        \ifnum\coordindex>#2\fi
    }
}}

\usepackage{wrapfig}
\newsavebox{\wrapquestionfigure}
\opencutright
\headers{Space-Time Form of the Schrödinger Equation}{S. Hain and K. Urban}

\title{An Ultra-Weak Space-Time Variational Formulation for the Schrödinger Equation\thanks{Submitted \today.
\funding{This work has partly been supported by the center of competence of the University of Stuttgart and Ulm University entitled \emph{IQ$^\text{ST}$--Integrated Quantum Science and Technology}.}}}

\author{Stefan Hain\thanks{
	Ulm University, 
	Institute of Numerical Mathematics, 
	Helmholtzstr.\ 18, 89081 Ulm (Germany), 
	(\email{stefan.hain@alumni.uni-ulm.de}, \email{karsten.urban@uni-ulm.de})}
\and
	Karsten Urban\footnotemark[2]}	

\begin{document}
\maketitle

\begin{abstract}
We present a well-posed ultra-weak space-time variational formulation for the time-dependent version of the linear Schrödinger equation with an instationary Hamiltonian. We prove optimal inf-sup stability and introduce a space-time Petrov-Galerkin discretization with optimal discrete inf-sup stability. 

We show norm-preservation of the ultra-weak formulation. The inf-sup optimal Petrov-Galerkin discretization is shown to be asymptotically norm-preserving, where the deviation is shown to be in the order of the discretization. In addition, we introduce a Galerkin discretization, which has suboptimal inf-sup stability but exact norm-preservation.

Numerical experiments underline the performance of the ultra-weak space-time variational formulation, especially for non-smooth initial data.
\end{abstract}

%

\section{Introduction}
\label{Sec:1}

We consider the inhomogeneous time-dependent linear Schrö\-dinger equation in \emph{atomic units} on a time interval ${I} := (0,T)$, $T > 0$, and a bounded spatial domain $\Omega \subsetneq \mathbb{R}^{3N}$ ($N \geq 1$ being the number of particles) with smooth boundary $\Gamma := \partial \Omega$, 
\begin{align}\label{eq:SchroedingerEquation}
	\left\{
	\begin{aligned}
		\mathrm{i} {\textstyle\frac{\partial}{\partial t}} u(t, x) 
			+  {\textstyle\frac{1}{2}} \Delta_x u(t,x) 
				-  \mathcal{V}(t,x)\, u(t,x) 
			&= g(t,x), \quad 
			&&(t,x) \in {I} \times \Omega, \\
		u(t,x) 
			&= 0,  
			&&(t,x) \in {I} \times \Gamma, \\
		u(0,x) 
			&= u_0, 
			&&x \in \Omega,
	\end{aligned}
	\right.
\end{align} 
 where $g: \Omega_T:={I} \times \Omega \rightarrow \mathbb{C}$ is an inhomogeneous right-hand side, $u_0$ some initial state and $\mathcal{V}: \Omega_T \rightarrow \mathbb{R}$ a real-valued potential. Usually, the problem is formulated on the full space $\mathbb{R}^{3N}$; however, since one can expect that the solution decreases very fast as $|x|\to\infty$ provided that the initial condition $u_0$ and the inhomogeneous right-hand side $g$ do so, we limit ourself to the case of a bounded (but \enquote{sufficiently large}) domain $\Omega \subsetneq \mathbb{R}^{3N}$. On the other hand, we stress the fact that certain parts of the presented material can also be extended to the full space, i.e., $\Omega = \mathbb{R}^{3N}$, with some technical modifications though, \cite{Hain}.
 
The time-dependent Schrödinger equation \eqref{eq:SchroedingerEquation} describes the time evolution of a quantum mechanical state represented by its solution $u$, which is also called \emph{wave function}. It is worth mentioning that unitary time-propagation (which is important from a physical point of view) requires either $g \equiv 0$ or a right-hand side, which depends on the wave function, i.e., $g(u)$. Still, we consider a right-hand side of the form displayed in the first equation in \eqref{eq:SchroedingerEquation}, i.e., $g\not\equiv 0$, but independent of $u$. This setting will be useful for  model reduction of parameterized versions of the Schrödinger equation. Model reduction is also one reason why we are interested in finding a \emph{well-posed weak (or variational) formulation} of \eqref{eq:SchroedingerEquation} in the following sense: Determine Hilbert spaces $\U$, $\V$ of functions and a sesquilinear form $b:\U\times\V\to\C$ such that given an \emph{arbitrary} functional $g\in\V'$, the variational problem 
\begin{equation}\label{eq:VarFormGen}
	b(u,v) = g(v)\quad\forall v\in\V,
\end{equation}
admits a unique solution $u^*\in\U$ (continuously depending on $g$ and $u_0$), which solves  \eqref{eq:SchroedingerEquation} in some appropriate weak sense. Consequently, the spaces $\U$ and $\V$ consist of functions in space and time, i.e., \eqref{eq:VarFormGen} is a \emph{space-time variational form}. 

Of course, there is a huge amount of literature concerning well-posedness results for the Schrödinger equation, but we are not aware of space-time variational approaches as \eqref{eq:VarFormGen}. In fact,  in \cite{SemilinearSchroedinger}, the time-dependent linear and nonlinear Schrödinger equation was treated by using semi-group theory, excluding time-dependent potentials, however. 
Time-dependent Hamiltonians are considered e.g.\ in \cite{ExternalPotential, DautrayLions, LionsMagenes} and \cite{Polizzi2,Polizzi3,Polizzi1, MSprengel, Gabriele1, Gabriele2}, where the multi-particle Schrödinger equation was formulated as a system of nonlinear single-particle Schrödinger equations. However, the results in \cite{Polizzi2,Polizzi3,Polizzi1} do not yield a space-time formulation as in \eqref{eq:VarFormGen}. On the other hand, \cite{MSprengel, Gabriele1, Gabriele2} point towards a well-posed weak space-time formulation, but these papers assume that the regularity of the wave function $u(t,\cdot)$ as a function in space grows for increasing time $t$ (see \cite[Thm.\ 3.13]{MSprengel},  \cite[Thm.\ 13]{Gabriele1}, \cite[Thm.\ 1]{Gabriele2}). However, as the Schrödinger equation (similar to wave or transport equations and in contrast to the heat equation) has no regularization effect, the (possibly low) regularity of the initial condition $u_0$ (and of a right-hand side $g$) is inherited for all times. We are particularly interested in minimal regularity requirements, which exceeds the findings in  \cite{MSprengel, Gabriele1, Gabriele2}.

To the best of our knowledge, a space-time formulation of the Schrödinger equation was first considered in \cite{SpacetimeDPG,CMollet}, however, for a single free particle (i.e., for a \emph{stationary} Hamiltonian without potential). In \cite{CMollet}, an approach analogously to the parabolic case was considered for the space-time variational formulation as well as its discretization. However, well-posedness and stability for the discretization have not been presented there. In \cite{SpacetimeDPG}, a well-posed ultra-weak formulation using a (nonconforming) discontinuous Petrov-Galerkin method was presented. This approach is somehow related to the present paper since it also addresses optimal stability (see below). Instead, we follow a path which was earlier presented for the transport equation in \cite{BrunkenSmetanaUrban,DahmenTransport} and for the wave equation in \cite{henning2021weak}, see also \cite{MR4387197,MR4283879}. The reason is that in particular  w.r.t.\ model reduction, we prefer a conforming discretization.

\subsection*{Outline}
This paper is organized as follows: In Section \ref{Sec:2}, we recall some facts on inf-sup stable variational problems and discuss why the known setting as in the parabolic case is not appropriate for the Schrödinger equation. Then we present an optimally inf-sup stable ultra-weak variational formulation there. Section \ref{Sec:3} is devoted to stable Petrov-Galerkin discretizations and suggests an adequate discretization specifically for the Schrödinger equation. In Section \ref{Sec:4}, we focus on the arising finite-dimensional algebraic system, which turns out to be of tensor product form. Finally, we present results of numerical experiments in Section \ref{Sec:5} in particular w.r.t.\ the numerical performance of the ultra-weak variational formulations -- especially for problems with non-smooth data.

\subsection*{Notation}
We abbreviate $H := L_2(\Omega; \mathbb{C})$ and always identify $L_2$-spaces with their duals. Then, we obtain Gelfand triples in space, i.e., $V := H_0^1(\Omega; \mathbb{C}) \hookrightarrow H \cong H' \hookrightarrow H^{-1}(\Omega; \mathbb{C}) := V'$, where all embeddings are continuous and dense. The spaces $H$ and $V$ are equipped with the usual inner products and norms. As usual, we identity the duality pairing $\langle u,v \rangle_{V' \times V}$ with the inner product $(u,v)_H$, if $u\in H$.  Finally, we set $\H:=L_2(I;H)\cong L_2(\Omega_T; \C)$.

%

\section{Variational Formulations of the Schrödinger Equation}
\label{Sec:2}

We start by recalling some known facts on the well-posedness of variational formulations and discuss possible approaches towards the Schrödinger equation.

\subsection{Inf-sup-theory}
\label{Subsec:InfSup}
 The well-posedness of \eqref{eq:VarFormGen} is described by the following well-known fundamental statement. 

\begin{theorem}[Ne\v{c}as Theorem, e.g.\ {\cite[Thm.\ 3.1]{necas2012direct}}]\label{thm:Necas}
	Let $\U$, $\V$ be Hilbert spaces, let $g\in\V'$ be given and $b:\U\times\V\to\C$ be a \emph{bounded} sesquilinear form, i.e., 
	\begin{align}
		\tag{C.1}\label{C1}
		&\exists\; \gamma<\infty:\quad
		|b(u,v)[ \le \gamma \| u\|_\U\, \| v\|_\V,
		\quad\text{for all } u\in\U, v\in\V.
	\end{align}
	Then, the variational problem \eqref{eq:VarFormGen} admits a unique solution $u^*\in\U$, which depends continuously on the data $g\in\V'$ if and only if
	\begin{align}
		\tag{C.2}\label{C2}
		&\beta:=\inf_{u\in\U} \sup_{v\in\V} \frac{|b(u,v)|}{\| u\|_\U\, \| v\|_\V}>0 
		&&\text{(inf-sup condition)};
		\\
		\tag{C.3}\label{C3}
		&\forall\, 0\ne v\in\V\quad \exists\, u\in\U: \quad b(u,v)\ne 0
		&&\text{(surjectivity)}.
	\end{align}
		\vskip-15pt\qed
\end{theorem}

The inf-sup constant $\beta$ (or some lower bound) also plays a crucial role for the numerical approximation of the solution $u\in\U$ since it enters the relation of the approximation error and the residual (by the Xu-Zikatanov lemma \cite{MR1971217}, see also \eqref{eq:bestApprox} below). This motivates our interest in the size of $\beta$: the closer to unity, the better.

A standard tool (at least) for (i) proving the inf-sup-stability in (C.2), (ii) stabilizing finite-dimensional discretizations, and (iii) getting  sharp bounds for the inf-sup constant, is to determine the so-called \emph{supremizer}. To define it, let $b:\U\times\V\to\R$ be a generic bounded bilinear form and $0\ne u\in\U$ be given. Then, the \emph{supremizer} $s_u\in\V$ is defined as the unique solution of the variational problem
\begin{equation}\label{eq:supremizer}
	(s_u,v)_\V = b(u,v)\qquad \forall v\in\V.
\end{equation}
It is easily seen that
\begin{equation}\label{eq:supremizer2}
	\sup_{v\in\V} \frac{|b(u,v)|}{\| v\|_\V} = \sup_{v\in\V} \frac{|(s_u,v)_\V|}{\| v\|_\V} = \| s_u\|_\V,
\end{equation}
which justifies the name \emph{supremizer}. Then, $\beta=\inf\limits_{u\in\U} \frac{\| s_u\|_\V}{\| u\|_\U}$.

\subsection{A too naive approach}
\label{SubSec:Var}
A straightforward approach to derive a space-time variational formulation of the Schrödinger equation is described in \cite[Ex.\ 6.4.6]{CMollet}, see also \cite[Thm\ 12, 13]{Gabriele1}. For simplicity, we restrict ourselves in this subsection to the case $u_0=0$, but the arguments here can easily be extended to inhomogeneous initial conditions as well. The (too naive) idea is to follow the path of space-time variational formulations of the heat equation as e.g.\ in \cite{SchwabStevenson,PateraUrbanImprovedErrorBound} and setting\footnote{$H_{\{0\}}^1({I}):=\{ \theta\in H^1(I): \theta(0)=0\}$ and $H_{\{T\}}^1({I}):=\{ \theta\in H^1(I): \theta(T)=0\}$ for later reference.} $\U := L_2\left( {I} ; V \right) \cap H^1_{\{0\}}\left( {I} ; V' \right)$ as well as $\V := L_2\left( {I} ; V \right)$ with norms $\Vert u \Vert_{\U}^2 := \Vert u \Vert_{L_2({I} ; V)}^2 + \Vert \partial_t u \Vert_{L_2({I} ; V')}^2$ and $\Vert v \Vert_{\V} := \Vert v \Vert_{L_2({I} ; V)}$. This yields a problem of the form \eqref{eq:VarFormGen} with 
\begin{align*}
	b(u, v) :=
		\int_0^T\
		\big[
		\mathrm{i} \langle \partial_t u(t), v(t) \rangle_{V' \times V} 
		-  a_t(u(t), v(t))\big] dt, 
\end{align*}
and $g(v) := \int_0^T \!\!\langle g(t) ,v(t) \rangle_{V' \times V} dt$, where the sesquilinear form $a_t: V\times V\to\C$ is defined by\footnote{$\cV(t)\equiv \cV(t,\cdot)$} 
\begin{align}\label{eq:at}
	a_t(\phi,\varphi):= \textstyle{\frac12} (\nabla_x \phi, \nabla_x \varphi)_H 
		+ (\cV(t)\phi,\varphi)_H, 
		\quad \phi, \varphi\in V.
\end{align}
Since the potential is real-valued and assuming that it is bounded and non-negative\footnote{The assumption of non-negativity is posed here just for convenience in order to show that a standard variational form does not yield a well-posed problem even in this simple case of a non-negative potential.}, it is easily seen that $a_t(\cdot,\cdot)$ is bounded, hermitian and coercive. Then, it is  straightforward to show that $b(\cdot,\cdot)$ is bounded, i.e., \eqref{C1}. In view of Theorem \ref{thm:Necas}, one would need to show  surjectivity \eqref{C3} and the inf-sup condition \eqref{C2}. The latter one, however, fails in general. 

To get an idea, why the inf-sup condition might fail, we consider the supremizer as in \eqref{eq:supremizer} for the case $\V=L_2(I;V)$. We denote by $A(t) := -\frac{1}{2} \Delta_x+\cV(t) \in \mathcal{L}(V, V')$ the operator associated  to $a_t(\cdot, \cdot)$, in the sense that $\langle A(t) \phi, \varphi \rangle_{V' \times V} := a_t(\phi,\varphi)$ for all $\phi,\varphi \in V$. Then, the supremizer $s_u$ for some $0\ne u\in\U$ is easily seen to read $s_u(t)=\mathrm{i}\,A(t)^{-1}\dot u(t)+u(t)$, which is quite similar to the parabolic case in \cite[App.\ A]{SchwabStevenson} with the only difference of the appearance of the complex unit $\mathrm{i}$ -- which actually causes the problem. In fact, in view of \eqref{eq:supremizer2}, we have for any $u\in\U$ that 
\begin{align*}
	\| s_u\|_\V^2
	&
	= ( \mathrm{i}\, A^{-1}\dot u+u, \mathrm{i}\,A^{-1}\dot u + u )_\V
	= \| A^{-1}\dot u \|_\V^2 + \| u\|_\V^2 - 2\, \Im (A^{-1}\dot u,u)_\V.
\end{align*}
We would now need to find a $\beta>0$ such that $\| s_u\|_\V\ge\beta\| u\|_\U$ (i.e., strictly positive for $u\ne 0$), which turns out to be impossible for arbitrary $u\ne 0$. The difference to the heat equation lies in the last term, which for the heat equation reads $-2(A^{-1}\dot u,u)_\V=\| u(0)\|_H^2$ (instead of the imaginary part). This term is thus non-negative (and the whole term strictly positive for $u\ne 0$), whereas we were not able to bound $ - 2\, \Im (A^{-1}\dot u,u)_\V$ from below; just using Cauchy-Schwarz and Young's inequality for this term yields $\| s_u\|_\V\ge 0$, which is clearly useless. 
We note, that this observation of not being able to find a positive lower bound for $\| s_u\|_\V$ is not just a theoretical consideration. In fact, numerical experiments for a corresponding Petrov-Galerkin method clearly showed instabilities, see \cite{Hain}.

\subsection{An optimally inf-sup stable ultra-weak variational form}
As already pointed out earlier, the Schrödinger equation does not have a smoothing effect, rough initial data or right-hand sides imply a rough solution for all times. Since this is quite similar to transport and wave equations, it seems natural to follow the same path that has successfully lead to optimally stable variational formulations in those cases, namely \emph{ultra-weak} formulations, see \cite{BrunkenSmetanaUrban,DahmenTransport,henning2021weak} (compare also the discontinuous Petrov-Galerkin setting e.g.\  in \cite{SpacetimeDPG}).

The idea is to multiply \eqref{eq:SchroedingerEquation} with sufficiently smooth test functions $v:\bar{I}\times\bar\Omega\to\C$ and to perform integration by parts w.r.t.\ to both space and time, i.e., 
\begin{align*}
	&(g,v)_{\H}
	= \int_I\int_\Omega g(t,x)\, \overline{v(t,x)}\, dx\, dt\\
	&\kern+5pt= \mathrm{i} \int_I\int_\Omega \dot{u}(t,x)\, \overline{v(t,x)}\, dx\, dt
		+\int_I\int_\Omega \Big( \textstyle{\frac12} \Delta_x u(t,x)-\cV(t,x)\, u(t,x)\Big) \overline{v(t,x)}\, dx\, dt\\
	&\kern+5pt= \mathrm{i} (u(T), v(T))_{H} - \mathrm{i} (u(0), v(0))_{H)} 
		+ (u, (\mathrm{i} \partial_t + \textstyle{\frac12} \Delta_x-\cV(t)) v)_{\H} \\
	&\kern+5pt= - \mathrm{i} (u_0, v(0))_{H} 
		+ (u, (\mathrm{i} \partial_t + \textstyle{\frac12} \Delta_x-\cV(t)) v)_{\H} 
\end{align*}
if $v(T)=0$. This motivates an ultra-weak variational formulation \eqref{eq:VarFormGen} of \eqref{eq:SchroedingerEquation} by defining 
\begin{subequations}
	\label{eq:veryweak}
	\begin{align}
		b(u,v) &:= (u, S^*v)_{\H)}  
			:= (u, \mathrm{i}\,\partial_t v + \textstyle{\frac12} \Delta_x v - \cV(t) v)_{\H}, 
			\label{eq:veryweak:1}\\
		g(v)	&:= \langle g,v\rangle_{\V'\times\V}+ \mathrm{i}\, (u_0, v(0))_{H},  	
			\label{eq:veryweak:2}
	\end{align}
\end{subequations}
where $S^*$ will be shown later to be the adjoint of the Schr\"odinger operator $S:=\mathrm{i}  \ts{\frac{\partial}{\partial t}}+ \ts{\frac{1}{2}} \Delta_x- \mathcal{V}(t)$ in an appropriate sense.

We want to show that this setting admits a well-posed problem of the form \eqref{eq:VarFormGen} if we define $\U$ and $\V$ appropriately. In order to do so, we are now going to follow \cite{DahmenTransport} and start by interpreting \eqref{eq:SchroedingerEquation} for $u_0=0$ in the classical sense, i.e., pointwise by setting
\begin{align*}
	S_{\circ} u (t,x) &:= \mathrm{i}  \ts{\frac{\partial}{\partial t}} u(t, x) 
		+ \ts{\frac{1}{2}} \Delta_x u(t,x) 
		- \mathcal{V}(t,x)\, u(t,x),
		\quad (t,x)\in\Omega_T=I\times\Omega, 
\end{align*}
as well as $g_{\circ}(t,x):=g(t,x)$. Homogeneous initial and boundary conditions are associated to $S_\circ$ as essential boundary conditions. Then, \eqref{eq:SchroedingerEquation} reads $S_{\circ} u = g_{\circ}$ in the space $C(\Omega_T;\C)$. We denote by\footnote{$C_{\{0\}}(\bar{I};\C) :=\{\theta\in C(\bar{I};\C):\, \theta(0)=0\}$ and $C_0(\bar\Omega;\C):=\{ \phi\in C(\bar\Omega;\C):\, \phi(x)=0, x\in\Gamma\}$.}
\begin{align}
	D(S_{\circ}) 
	&:= \{ u: \overline\Omega_T\to\C:
	&&\hskip-18pt S_{\circ} u\in C(\overline\Omega_T;\C); u(t,x)=0 \text{ for } (t,x) \in {I} \times \Gamma; \nonumber\\
	&&&\hskip-18pt u(0,x)= 0 \text{ for } x \in \Omega\} \nonumber \\
	&=: C^{1,2}_{\{0\},\Gamma}(\Omega_T;\C)
	&&\hskip-21pt = 
		[C^1(I;\C) \cap C_{\{0\}}(\bar{I};\C)] \otimes [ C^2(\Omega;\C) \cap C_0(\bar\Omega;\C)],
		\label{eq:C120Gamma}
\end{align}
the classical domain of $S_{\circ}$ incorporating homogeneous initial and boundary conditions. Of course, the regularity of the potential $\cV$ is crucial and we will come back to this topic later, see Theorem \ref{Thm:stablecont} below. In \eqref{eq:C120Gamma} the superscripts denote the regularity in time and space, respectively, whereas the index $\{0\}$ refers to homogeneous temporal conditions for $t=0$ and $\Gamma$ for homogeneous Dirichlet conditions on $\Gamma$. We will use a similar notation below also for $t=T$. For the above setting it is necessary that $\cV\in C(\overline\Omega_T;\C)$.

For the rest of the paper, we shall always set 
\begin{align}\label{eq:U}
	\U:=
	L_2(I;H) = \H
\end{align}
with its usual topology and we define the formal adjoint $S_\circ^*$ of $S_\circ$ by
\begin{align*}
	(S_\circ u, v)_\U = (u, S_\circ^* v)_\U
	\quad\text{ for all } u,v\in C^\infty_0(\Omega_T;\C).
\end{align*}
Following \cite{DahmenTransport}, we need to verify the following conditions
\begin{compactenum}
	\item[($S^*1$)] $S_\circ^*$ is injective on the dense subspace $D(S_\circ^*)\subset\U$ and 
	\item[($S^*2$)] the range $R(S_\circ^*)\hookrightarrow\U$ is densely imbedded.
\end{compactenum}
In fact, the properties ($S^*1$) and ($S^*2$) ensure that $\| v\|_\V := \| S_\circ^* v\|_\U$ is a norm on $D(S_\circ^*)$. Then, 
\begin{align}\label{Def:V}
	\V := \mathrm{clos}_{\|\cdot\|_\V} \{D(S_\circ^*)\} \subset\U,
	\quad
	(v,w)_\V := (S^* v, S^*w)_\U,
	\,\,\, v, w\in\V,
\end{align}
is a Hilbert space with induced norm
\begin{align}\label{eq:normV}
	\| v\|_\V := \| S^* v\|_\U,
\end{align}
where $S^*$ is the continuous extension of $S^*_\circ$ from $D(S_\circ^*)$ to $\V$. As we shall see below, this framework yields a well-posed ultra-weak variational form of the Schrödinger equation. To see this, we need to detail the involved ingredients. We start by determining the formal adjoint $S_{\circ}^*$.

\begin{proposition}
\label{proposition:FormalAdjointSchroedinger}
	The formal adjoint  Schrödinger operator reads $S_{\circ}^* = S_{\circ}$.
\end{proposition}

\begin{proof}
	Let $u, v \in C_0^{\infty}(\Omega_T;\C)$. Then, using integration by parts in both space and time variables, it holds 
\allowdisplaybreaks
\begin{align*}
(S_{\circ} u, v)_{\U} 
	&= \left( \mathrm{i} \partial_t u 
		+ \textstyle{\frac{1}{2}} \Delta_x u 
		- \mathcal{V}(t,x) u, \; v \right)_{\H} \\
	&\kern-20pt=  
		\int_0^T\!\!\!\int_{\Omega} \left[
			\mathrm{i} \partial_t u(t,x) \; \overline{v(t,x)}  
	 		+ \textstyle{\frac{1}{2}} \Delta_x  u(t,x) \overline{v(t,x)} 
			- \mathcal{V}(t,x) u(t,x) \overline{v(t,x)} \right] dx \; dt \\
	&\kern-20pt= \int_0^T\!\!\!\int_{\Omega}
		\left[ 
			u(t,x) \; \overline{\mathrm{i} \partial_t v(t,x)}
			+ u(t,x) \overline{\textstyle{\frac{1}{2}} \Delta_x v(t,x)} 
			- \mathcal{V}(t,x) u(t,x) \overline{v(t,x)} 
			\right] dx \; dt \\
	&\kern-20pt= \left(u, \mathrm{i} \partial_t v + \textstyle{\frac{1}{2}} \Delta_x v - \mathcal{V}(t,x) v \right)_{\H} 
	= (u, \; S_{\circ}v)_{\U},
\end{align*}
since $\cV$ is real-valued, which proves the claim.
\end{proof}

Next, we need to determine domain and range of $S_{\circ}^*$. To this end, we first note, that due to the integration by parts w.r.t.\ time used in Proposition \ref{proposition:FormalAdjointSchroedinger}, the initial conditions in $S_{\circ}$ are transformed to \emph{terminal} conditions for $S_{\circ}^*$. In view of \eqref{eq:C120Gamma}, we thus obtain
\begin{align}
\label{eq:ClassicalDomainFormalAdjointSchroedinger}
	D(S_{\circ}^*) = C^{1,2}_{\{T\},\Gamma}(\Omega_T;\C)
	=  [C^1(I;\C) \cap C_{\{T\}}(\bar{I};\C)] \otimes [ C^2(\Omega;\C) \cap C_0(\bar\Omega;\C)].
\end{align}
The range $R(S_\circ)$ in the classical sense then reads 
\begin{align*}
	R(S_\circ) = C(\overline\Omega_T;\C),
\end{align*} 
so that ($S^*2$) holds. Therefore it remains to verify that ($S^*1$) is satisfied.

\begin{proposition}
\label{proposition:FormalAdjointInjective}
The formal adjoint $S_{\circ}^*$ is injective on $D(S_{\circ}^*)$.
\end{proposition}

\begin{proof}
Let $u\in C^{1,2}_{\{T\},\Gamma}(\Omega_T)$ such that $S_{\circ}^*u=0$. Then, for any $\phi \in V$ and for $s \in{I}$ we have
\begin{align*}
	\mathrm{i} \left( \partial_t \psi(s),  \phi \right)_{H} 
		- \textstyle{\frac{1}{2}} \left( \nabla_x \psi(s)  , \nabla_x\phi \right)_{H} 
		- \left( \mathcal{V}(s) \psi(s), \phi \right)_{H} 
		= 0.
\end{align*}
We choose $\psi(s)$ as test function and obtain 
\begin{align}\label{eq:Equation4}
	\mathrm{i} \left( \partial_s \psi(s) , \psi(s) \right)_{H} 
	- \textstyle{\frac{1}{2}}\left( \nabla_x \psi(s) , \nabla_x \psi(s)\right)_{H} 
	- \left( \mathcal{V}(s) \psi(s), \psi(s)\right)_{H} = 0.
\end{align}
Note, that $ \textstyle{\frac{1}{2}}\left( \nabla_x \psi(s) , \nabla_x \psi(s)\right)_{H} - \left( \mathcal{V}(s) \psi(s), \psi(s)\right)_{H} \in\R$.
Thus, integrating over $[t,T]$ for $t \in{I}$ and taking the imaginary part of \eqref{eq:Equation4} yields
\begin{align*}
	0
	&= \int_t^T{ \left( \partial_s \psi(s), \psi(s) \right)_{H} ds}  
	= \ts{\frac{1}{2}} \int_t^T{ \ts{\frac{d}{ds}} \Vert \psi(s) \Vert_{H}^2 ds} 
	= -\ts{\frac{1}{2}} \Vert \psi(t) \Vert_{H}^2,
\end{align*}
since $\psi(T)=0$, which implies $\Vert \psi(t) \Vert_{H} = 0$ for all $t \in{I}$ and proves the claim.
\end{proof}

\begin{theorem}\label{Thm:stablecont}
	Let $\cV\in L_\infty(\Omega_T;\R)$ and $g\in\V'$. Moreover, let $\U$, $\V$, $b(\cdot,\cdot)$, and $g(\cdot)$ be defined as in \eqref{eq:U}, \eqref{Def:V} and \eqref{eq:veryweak}, respectively. 
	Then, the variational problem \eqref{eq:VarFormGen} admits a unique solution $u^*\in\U$. 
	In particular,
	\begin{equation}
		\beta := 
		\inf_{u\in\U} \sup_{ v\in\V} \frac{b(u, v)}{\| u\|_\U\, \|  v\|_{\V}} 
		= \sup_{u\in\U} \sup_{ v\in\V} \frac{b(u, v)}{\| u\|_\U\, \|  v\|_{\V}} 
		= 1.
	\end{equation}
\end{theorem}
\begin{proof}
	We are going to show the conditions (C.1)-(C.3) of Theorem \ref{thm:Necas} above.\\
	(C.1) \emph{Boundedness}: Let $u\in\U$, $ v\in\V$. By Cauchy-Schwarz' inequality and assumption, we get $\|S^* v\|_\U = \| \mathrm{i}\partial_t v + \textstyle{\frac12} \Delta_x v - \cV v\|_\U \le \| \partial_t v\|_\U + \textstyle{\frac12} \| \Delta_x v\|_\U + \|\cV\|_{L_\infty} \| v\|_\U \le C \| u\|_{H^{1,2}(\Omega_T)} <\infty$.\footnote{We abbreviate $H^{s,k}(\Omega_T):= H^s(I; H^k(\Omega)) \cong H^s(I)\otimes H^k(\Omega)$.}  
	Again by Cauchy-Schwarz' inequality, we have
	\begin{align*}
		|b(u, v)|
		&= |(u, \mathrm{i}\partial_t v + \textstyle{\frac12} \Delta_x v - \cV v)_\U|
		= |(u, S^*v)_\U|
		\le \| u\|_\U \, \| S^* v\|_\U
		= \|u\|_\U\, \| v\|_{\V},
	\end{align*}
	i.e., the continuity constant is unity. \\
	(C.2) \emph{Inf-sup}: Let $0\ne u\in\U$ be given. We consider the supremizer $s_u\in \V$ defined as $(s_u,  v)_{\V} = b(u, v) = (u, S^*v)_\U$ for all $ v\in \V$. Since  by definition of the inner product 
	$
	(s_u,  v)_{\V} 
	= ( S^* s_u, S^*v)_\U 
	$ 
	for all $ v\in \V$ we get $S^* s_u=u$ in $\U$. Then, by \eqref{eq:supremizer2}, 
	\begin{align*}
	\sup_{ v\in\V} \frac{|b(u, v)|}{\|  v\|_{\V}} 
	&= \sup_{ v\in\V} \frac{|(s_u, v)_{\V}|}{\|  v\|_{\V}} 
	= \| s_u\|_{\V}
	= \| S^* s_u\|_\U 
	= \| u\|_\U,
	\end{align*}
	i.e., $\beta=1$ for the inf-sup constant. \\
	(C.3) \emph{Surjecitivity}: Let $0\ne v\in\V$ be given. Then, there is a sequence $(v_n)_{n\in\N}\subset C^{1,2}_{\{T\},\Gamma}(\Omega_T)$ with $v_n\not=0$, converging towards $v$ in $\V$. Since $S_\circ^*$ is an isomorphism of $C^{1,2}_{\{T\},\Gamma}(\Omega_T)$ onto $C([0,T];H)$, there is a unique $u_n:=S^*_\circ v_n\in C([0,T];H)$. Hence $0\not= \| u_n\|_{C([0,T];H)}$. Possibly by taking a subsequence,  $(u_n)_{n\in\N}$ converges to a unique limit $u_v\in L_2(I;H)$. 
	We take the limit as $n\to\infty$ on both sides of $u_n= S^* v_n$ and obtain $0\ne u_v=S^*v\in L_2(I;H)=\U$. 
	Finally, $|b(u_v,v)| = |(u_v, S^*v)_\U| = |(u_v, u_v)_\U| = \| u_v\|_\U^2> 0$, which proves surjectivity and concludes the proof.
\end{proof}

\begin{remark}
	The essence of the above proof is the fact that $\U$ and $\V$ are related as $\U=S^*(\V)$ and noting that $S$ is self-adjoint up to initial versus terminal conditions.
\end{remark}

\begin{remark}\label{Rem:VVcirc}
	For later reference, we note, that
	\begin{align}\label{eq:Vcirc}
		\V \subseteq H^1_{\{T\}}(I;\C) \otimes [H^1_0(\Omega;\C)\cap H^2(\Omega;\C)]
			=: \V_\circ.
	\end{align}
	Moreover, for $v\in\V$, we have
	\begin{align*}
		\| v\|_\V
		&= \| S^* v\|_\U = \| \mathrm{i}\partial_t v + \textstyle{\frac12} \Delta_x v - \cV v\|_\U \\
		&\le \| \partial_t v\|_\U + \textstyle{\frac12} \| \Delta_x v\|_\U + \|\cV\|_{L_\infty} \| v\|_\U 
		\le C \| u\|_{H^{1,2}(\Omega_T)} =: C \| u\|_{\V_\circ}.
	\end{align*}
\end{remark}

\begin{corollary}
	If the unique solution $u^*\in\U$ is in $C^{1,2}(\Omega_T;\C)$, then it satisfies \eqref{eq:SchroedingerEquation} pointwise.
\end{corollary}
\begin{proof}
	The proof is immediate by reverting the integration by parts.
\end{proof}

\section{Petrov-Galerkin Discretization}
\label{Sec:3}
We determine a numerical approximation to the solution of a variational problem of the general form \eqref{eq:VarFormGen}.  To this end, one chooses finite-dimensional trial and test spaces, $\U_\delta\subset\U$, $\V_\delta\subset\V$, respectively, where $\delta$ is a discretization parameter to be explained later.  For convenience, we assume that their dimension is equal, i.e., $\cN_\delta:=\dim\U_\delta=\dim\V_\delta$. The Petrov-Galerkin method then reads
\begin{align}\label{eq:var-disc}
	\text{find } u_\delta\in\U_\delta:\quad
	b(u_\delta,v_\delta) =   g(v_\delta)
	\quad\text{for all } v_\delta\in \V_\delta.
\end{align}
As opposed to the coercive case, the well-posedness of \eqref{eq:var-disc} is not inherited from that of \eqref{eq:VarFormGen}. In fact, in order to ensure uniform stability (i.e., stability independent of the discretization parameter $\delta$), the spaces $\U_\delta$ and $\V_\delta$ need to be appropriately chosen in the sense that the discrete inf-sup (or LBB -- Ladyshenskaja-Babu\v{s}ka-Brezzi) condition holds, i.e., there exists a $\beta_\circ>0$ such that
\begin{align}\label{eq:LBB}
	\beta_\delta
	&:=\inf_{u_\delta\in\U_\delta} \sup_{v_\delta\in\V_\delta} 
		\frac{|b(u_\delta,v_\delta)|}{\| u_\delta\|_\U\, \| v_\delta\|_\V}
	\ge \beta_\circ >0,
\end{align}
where the crucial point is that $\beta_\circ$ is independent of $\delta$. The size of $\beta_\circ$ is also relevant for the error analysis, since the Xu-Zikatanov lemma \cite{MR1971217} yields a best approximation result 
\begin{align}\label{eq:bestApprox}
	\| u^*-u^*_\delta\|_\U \le \frac{1}{\beta_\circ} \inf_{w_\delta\in\U_\delta}\| u^*-w_\delta\|_\U
\end{align}
for the \enquote{exact} solution $u^*$ of \eqref{eq:VarFormGen} and the \enquote{discrete}  solution $u_\delta^*$ of \eqref{eq:var-disc}. This is also the key for an optimal error/residual relation, which is important for a posteriori error analysis (also within the reduced basis method).

\subsection{A stable Petrov-Galerkin space-time discretization}
\label{sec:StablePG}
To properly discretize $\V$,
we consider the tensor product subspace $\V_\circ \subset \V$ introduced in \eqref{eq:Vcirc} which allows for a straightforward finite element discretization.\footnote{Of course, one could also use fully unstructured discretizations in space and time as opposed to a tensor product structure, \cite{Steinbach}. However, for efficiency of the solution of the arising algebraic systems, the tensor product structure turned out to be very usefull.} Hence, we look for a pair $\U_\delta\subset\U$ and $\V_\delta\subset\V_\circ$ satisfying \eqref{eq:LBB} with a possibly large inf-sup lower bound $\beta_\circ$, i.e., close to unity. Constructing such a stable pair of trial and test spaces is again a nontrivial task, not only for the Schrödinger equation. It is a common approach to choose some trial approximation space $\U_\delta$ (e.g.\ by splines) and then (try to) construct an appropriate according test space $\V_\delta$ in such a way that \eqref{eq:LBB} is satisfied. This can be done, e.g., by computing the supremizers for all basis functions in $\U_\delta$ and then define $\V_\delta$ as the linear span of these supremizers. However, this would amount to solve the original problem $\cN_\delta$ times, which is way too costly. We mention that this approach indeed works within the discontinuous Galerkin (dG) method, see, e.g., \cite{BTDeGh13,DemGop11}.
We will follow a different path, also used in \cite{BrunkenSmetanaUrban} for transport problems and in \cite{henning2021weak} for the wave equation. We first construct a test space $\V_\delta$ by a standard approach and then define a stable trial  space $\U_\delta$ in a second step. This implies that the trial functions are no longer `simple' splines but they arise from the application of the adjoint operator $S^*$ (which is here the same as the primal one except for initial/terminal conditions) to the test basis functions.

\subsubsection*{Finite elements in time.} 
We start with the temporal discretization. 
We choose some integer ${N_t}>1$ and set $\Dt:=T/{N_t}$. This results in a temporal \enquote{triangulation} 
\begin{align*}
	\cT_{\Dt}^\text{time}\equiv\{ t^{k-1}\equiv(k-1)\Delta t < t \le k\, \Dt \equiv t^k, 1\le k\le {N_t}\}
\end{align*}
in time. Then, we set 
\begin{align}\label{eq:RDt}
	R_\Dt := \Span\{ \varrho^1,...,\varrho^{{N_t}}\}\subset H^1_{\{T\}}(I;\R), 
\end{align}	
e.g.\ piecewise linear splines on $\cT_{\Dt}^\text{time}$ with standard modification at the right end point of $\bar I=[0,T]$. More general, we use a spline basis of degree $p_{\text{time}}\ge 1$. 

\subsubsection*{Discretization in space.} 
For the space discretization, we  choose any conformal finite element space
\begin{align}\label{eq:Zh}
	Z_h :=\Span\{ \phi_1,...,\phi_{{N_h}}\}\subset H^1_0(\Omega;\R)\cap H^2(\Omega;\R),
\end{align}
e.g.\ piecewise quadratic finite elements with homogeneous Dirichlet boundary conditions. For later reference, we choose a spline basis of degree $p_{\text{space}}\ge 2$.

In any case, we assume to have standard spline approximation results of the following form available
\begin{subequations}\label{eq:SplineApprox}
\begin{align}
	\inf_{r_\Dt\in R_\Dt} \| r-r_\Dt\|_{H^r(I)}
	&\le (\Dt)^{\ell+1-r}\, |r|_{H^{\ell+1}(I)},
	&& 0\le r\le\ell\le p_{\text{time}}, \\
	\inf_{z_h\in Z_h} \| z-z_h\|_{H^r(\Omega)}
	&\le h^{\ell+1-r}\, |z|_{H^{\ell+1}(\Omega)}, 	
	&& 0\le r\le\ell\le p_{\text{space}}. 
\end{align}
\end{subequations}

\subsubsection*{Complex-valued test and trial space in space and time.} 
Then, we define the test space as
\begin{align}\label{eq:Vdelta}
	\V_\delta &:= \Span\nolimits_\C(R_\Dt\otimes Z_h) \subset \V_\circ \subset \V, 
	\qquad \delta=(\Dt,h), \\
	&= \Span\nolimits_\C\{ \varphi_\nu := \varrho^k\otimes \phi_i:\, k=1,...,N_t,\, i=1,...,N_h, \nu=(k,i)\}
	\nonumber \\
	&=\bigg\{
		\sum_{\nu=(k,i)=(1,1)}^{(N_t,N_h)} c_\nu\, \varphi_\nu:\, c_\nu\in\C
		\bigg\}, \nonumber
\end{align}
which is a complex-valued tensor product space of dimension $\cN_\delta = N_t\, N_h$. Moreover, by \eqref{eq:SplineApprox} and Remark \ref{Rem:VVcirc}
\begin{align}\label{eq:ErrorOrder}
	\inf_{v_\delta\in\V_\delta} \| v-v_\delta\|_\V
	&\le C \inf_{v_\delta\in\V_\delta} \| v-v_\delta\|_{\V_\circ}
	= \cO ((\Dt)^k +h^{m-1}),
\end{align}
provided that $v\in H^{k+1,m+1}(\Omega_T;\R)$, $1\le k\le p_{\text{time}}-1$, $1\le m\le p_{\text{space}}-1$.\footnote{Use $r=1$, $\ell=k$ in time and $r=2$, $\ell=$ in space in \eqref{eq:SplineApprox}.}

The trial space $\U_\delta$ is constructed by applying the adjoint operator $S^*$ to each test basis function, i.e., for $\mu=(\ell,j)$ and $A=-\Delta$
\allowdisplaybreaks
\begin{align*}
	\psi_\mu 
	&:= S^*(\varphi_\mu) 
	=  S^*(\varrho^\ell\otimes \phi_j) 
	= (\mathrm{i}\partial_{t} -\textstyle{\frac12}\Delta_x + \cV)(\varrho^\ell\otimes \phi_j),
\end{align*}
i.e., $\U_\delta := S^*(\V_\delta) = \Span\{ \psi_\mu:\, \nu=1,...,\cN_\delta\}$. Since $S^*$ is an isomorphism of $\V$ onto $\U$, the functions $\psi_\nu$ are in fact linearly independent. A sample is shown in Figure \ref{Fig:Sample}.

\begin{center}
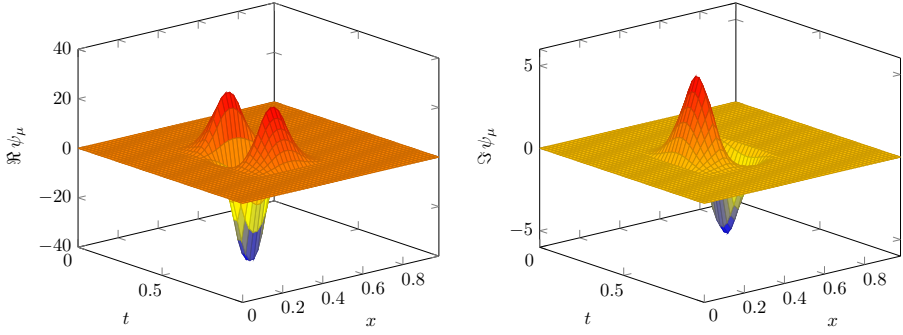
\begin{figure}[h!]
	\begin{tikzpicture}[scale = 0.7]
	  \begin{axis}[view={50}{20}, zmin = -40, zmax = 40, xlabel={$t$}, ylabel={$x$}, zlabel={$\Re\, \psi_{\mu}$}]
		  	\addplot3[surf] table [x=y, y=x,  z=z_real] {results/valuesTrialFunction_32.txt};
		  \end{axis}
	\end{tikzpicture} 
	  \hspace*{2mm}
	\begin{tikzpicture}[scale = 0.7]
	  \begin{axis}[view={50}{20}, zmin = -6, zmax = 6, xlabel={$t$}, ylabel={$x$}, zlabel={$\Im\, \psi_{\mu}$}]
		  	\addplot3[surf] table [x=y, y=x,  z=z_imag] {results/valuesTrialFunction_32.txt};
		  \end{axis}
	\end{tikzpicture}
\caption{Trial function for $I = \Omega = (0,1)$ (left: real part, right: imaginary part).}
\label{Fig:Sample}
\end{figure}
\end{center}

\begin{proposition}\label{Lem:LBB}
 	For the space $\V_\delta$ defined in \eqref{eq:Vdelta} and $\U_\delta:=S^*(\V_\delta)$, we have
	\begin{align*}
	\beta_\delta := 
	\inf_{u_\delta\in\U_\delta} \sup_{v_\delta\in\V_\delta} 
		\frac{|b(u_\delta,v_\delta)|}{\| u_\delta\|_\U\, \| v_\delta\|_\V} = 1.
\end{align*}
In particular, the Petrov-Galerkin problem \eqref{eq:var-disc} admits a unique solution $u_{\delta} \in \U_\delta$. 
\end{proposition}
\begin{proof}
	Let $0\ne u_\delta\in\U_\delta\subset\U$. Then, since $\U_\delta=S^*(\V_\delta)$ there exists a unique $z_\delta\in\V_\delta$ such that $S^*z_\delta=u_\delta$. Hence
\allowdisplaybreaks
	\begin{align*}
		\sup_{v_\delta\in\V_\delta}  \frac{|b(u_\delta,v_\delta)|}{\| u_\delta\|_\U\, \| v_\delta\|_\V}
		&\ge \frac{|b(u_\delta,z_\delta)|}{\| u_\delta\|_\U\, \| z_\delta\|_\V}
		= \frac{|(u_\delta,S^*z_\delta)_\U|}{\| u_\delta\|_\U\, \| z_\delta\|_\V}
		= \frac{|(u_\delta,u_\delta)_\U|}{\| u_\delta\|_\U\, \| B^*z_\delta\|_\U} \\
		&=\frac{ \|u_\delta\|^2_\U}{\| u_\delta\|_\U\, \| u_\delta\|_\U} = 1.
	\end{align*}
	On the other hand, by the Cauchy-Schwarz inequality, we have
\allowdisplaybreaks
	\begin{align*}
		\sup_{v_\delta\in\V_\delta}  \frac{|b(u_\delta,v_\delta)|}{\| u_\delta\|_\U\, \| v_\delta\|_\V}
		&= \sup_{v_\delta\in\V_\delta}  \frac{|(u_\delta, S^*v_\delta)_\U|}{\| u_\delta\|_\U\, \| v_\delta\|_\V}
		\le \sup_{v_\delta\in\V_\delta}  \frac{ \| u_\delta\|_\U\, \|S^*v_\delta\|_\U}{\| u_\delta\|_\U\, \| S^*v_\delta\|_\U} = 1,
	\end{align*}
	which proves the claim.
\end{proof}

\subsection{Approximation properties}
\label{sec:ApproximationProperties}

In view of \eqref{eq:bestApprox} and Proposition \ref{Lem:LBB} we can further detail the error as follows. Denoting by $v^*\in\V$ the unique element such that $S^*v^*=u^*$, we get by \eqref{eq:ErrorOrder} that 
\begin{align*}
	\| u^*-u^*_\delta\|_\U 
	&= \inf_{u_\delta\in\U_\delta}\| u^*-u_\delta\|_\U
	=  \inf_{v_\delta\in\V_\delta}\| S^*v^*-S^*v_\delta\|_\U \\
	&= \inf_{v_\delta\in\V_\delta}\| v^*- v_\delta\|_\V
	= \cO ((\Dt)^{k} +h^{m-1}),
\end{align*}
provided that $u^*\in H^{k+1,m+1}(\Omega_T;\C)$, $1\le k\le p_{\text{time}}-1$, $1\le m\le p_{\text{space}}-1$ and $\U_\delta\subset H^{k+1,m+1}(\Omega_T;\C)$ (which can be realized by choosing sufficiently smooth test functions).\footnote{This is achieved by using \eqref{eq:SplineApprox} with $r=1$, $\ell=k$ in time and $r=2$, $\ell=m$ in space.} 
This means that we obtain a desired rate of approximation by choosing the test space appropriately (as long as the solution $u^*$ is sufficiently smooth).

\subsection{Norm-preservation}
The absolute squared value of the wave function (its energy) at a given time is interpreted in quantum mechanics as a probability measure for finding a particle in a given place. Therefore, the energy should be a probability measure for all times, which is termed \emph{norm-preserving property}.

\begin{definition}\label{Def:NormPreserving}
	A function $\psi\in C(\overline{I}; H)$ is called \emph{norm-preserving}, if $\Vert \psi(t) \Vert _H = \Vert \psi(0) \Vert_H$ for all $t \in [0,T]$.
\end{definition}

\begin{remark}
	The solution $u^*$ and the discrete solution $u_\delta^*$, respectively, can only be norm-preserving if $g \equiv 0$ in \eqref{eq:SchroedingerEquation}. In fact, $g$ acts as an external force, so that $g \not\equiv 0$ would add additional energy.
\end{remark}

Since norm-preservation is a crucial property from an application point of view, we will discuss it within the space-time framework, both in the infinite and the finite-dimensional case. The space-time framework involves \emph{integration over time}, so that \emph{pointwise} properties like norm-preservation in the above sense is by no means obvious. 

This also holds in particular for a Petrov-Galerkin discretization. A more standard discretization is to first discretize in space (by a suitable method of choice), so that it remains to choose a time-stepping scheme. Those schemes can be constructed in such a way that they are energy-preserving, e.g.\ \cite{WDoerflerSchroedinger, GabrieleDissertation}. A Petrov-Galerkin method in space and time does not need to be a time-marching scheme. Hence, preservation from one time step to the next one cannot be an issue.

For the subsequent analysis, we need some assumptions, which we collect for brevity.

\begin{assumption}
\label{assumption:UnitaryTimeEvolution}
\begin{compactenum}[({A}.1)] 
	\item Let $g \equiv 0$.
	\item Assume that the initial condition is smooth, i.e., $u_0 \in V$.
	\item The potential $\mathcal{V}$ is assumed to have the following properties: it is 
	\begin{compactenum}
		\item \emph{uniformly bounded}, i.e., $\|\mathcal{V}(t)\phi\|_H \le C_{\mathcal{V}}\|\phi\|_H$ for all $\phi\in H$, $C_{\mathcal{V}}<\infty$;
		\item \emph{uniformly positive}, i.e., $(\cV(t)\phi,\phi)_H\ge \alpha_{\mathcal{V}}\|\phi\|_H^2$ for all $\phi\in H$, $\alpha_{\mathcal{V}}>0$;
		\item \emph{hermitian}, i.e., $(\cV(t)\phi,\varphi)_H = (\phi,\cV(t)\varphi)_H$ for all $\phi,\varphi\in H$;
		\item \emph{differentiable} w.r.t.\ $t$.
	\end{compactenum}
\end{compactenum}
\end{assumption}

We note that Assumption \ref{assumption:UnitaryTimeEvolution}(A.3) implies that $a_t(\phi,\phi)\in\R$ for all $\phi\in V$, which will be crucial in the sequel.

\begin{remark}
	Assumption \ref{assumption:UnitaryTimeEvolution}(A.3) implies that  the sesquilinear form $a_t$ defined in \eqref{eq:at} is hermitian, bounded and coercive on $V$ as well as weakly differentiable in time with bounded time-derivative $\dot{a}_t(\cdot, \cdot)$. 
\end{remark}

\subsubsection{Continuous case}
\label{sec:UnitaryTimeEvolutionCont}
Using Assumption \ref{assumption:UnitaryTimeEvolution}, it is well-known that the following variational problem admits a unique solution: find $u \in L_2(I;V) \cap H^1(I;V') \hookrightarrow C(\overline{I}; H)$ satisfying  $u(0) = u_0$ in $H$ such that
\begin{align}
\label{eq:UnitaryTimeEvolCont1}
	\int_0^T\!\!\! \{ \mathrm{i} \langle \partial_t u(t), v_1(t) \rangle_{V' \times V} \!-\!  a_t(u(t), v_1(t)) \} dt  
	+ \mathrm{i} \langle u(0),v_2 \rangle_{V \times V'} = \mathrm{i} \langle u_0,v_2 \rangle_{V \times V'}
	\hskip-4pt
\end{align}%
for  all test functions $v =(v_1,v_2)\in L_2(I;V) \times V'$, see e.g. \cite{DautrayLions, Hain}. Moreover, the solution $u$ is norm-preserving, which can easily be seen by choosing $v_1 = u$ and $v_2 = - \mathrm{i} u(0)$ in \eqref{eq:UnitaryTimeEvolCont1} and considering the imaginary part, which yields
\begin{align}
\label{eq:UnitaryTimeEvolCont2}
\int_0^T{ \frac{d}{dt} \Vert u(t) \Vert_H^2 \; dt} = \Vert u(T) \Vert_H^2 - \Vert u(0) \Vert_H^2 = 0,
\end{align}
which implies $\Vert u(T) \Vert_H = \Vert u(0) \Vert_H$.  Since \eqref{eq:UnitaryTimeEvolCont2} is valid for any $T>0$ we obtain the norm-preserving property of the analytical solution.  The norm-preserving property of the ultra-weak solution follows immediately by the following proposition:

\begin{proposition}
\label{prop:UnitaryTimeEvol1}
Let Assumption \ref{assumption:UnitaryTimeEvolution} be satisfied and assume that the weak solution of \eqref{eq:UnitaryTimeEvolCont1} is $u^* \in L_2(I;V) \cap H^1(I;V')$. Then $u^*$ is the (unique) weak solution of 
\begin{align}
\label{eq:UnitaryTimeEvolCont3}
(u, S^* v)_\U = \mathrm{i} (u_0, v(0))_{H}, \qquad \forall v \in \V.
\end{align}

\end{proposition}

\begin{proof}
Due to the assumed regularity of $u^*$, we can perform integration by parts in \eqref{eq:UnitaryTimeEvolCont1} and obtain
\begin{align*} 
	 \mathrm{i} \int_0^T\!\!{ \langle \partial_t u^*(t), v(t) \rangle_{V' \times V} \; dt} - \int_0^T{ a_t(u^*(t), v(t)) \; dt } 
	+ \mathrm{i} (u^*(0), v(0))_H
	&= (u^*, S^* v)_{\U},
\end{align*}
for all $v \in \V$.  Since $\V \subset L_2(I;V)$, we can use \eqref{eq:UnitaryTimeEvolCont1} with $v_1 = v$ and $v_2 = v(0)$ and obtain $(u^*, S^*v)_{\U} = \mathrm{i} (u_0, v(0))_H$, i.e., \eqref{eq:UnitaryTimeEvolCont3}.
\end{proof}

\begin{corollary}
Under the assumptions of Proposition \ref{prop:UnitaryTimeEvol1} the ultra-weak solution $u^*$ of \eqref{eq:UnitaryTimeEvolCont3} has additional regularity $u^* \in L_2(I;V) \cap H^1(I;V')$ and is norm-preserving. 
\end{corollary}
\begin{proof}
	Assumption \ref{assumption:UnitaryTimeEvolution} yields a smooth solution $u^* \in L_2(I;V) \cap H^1(I;V')$ of \eqref{eq:UnitaryTimeEvolCont1}, which by Proposition \ref{prop:UnitaryTimeEvol1} coincides with the unique ultra-weak solution of \eqref{eq:UnitaryTimeEvolCont3}. Hence, we obtain additional regularity and \eqref{eq:UnitaryTimeEvolCont2} yields norm preservation.
\end{proof}

\begin{remark}
We can assume weaker assumptions for $a_t(\cdot, \cdot)$ to obtain a unique solution of \eqref{eq:UnitaryTimeEvolCont1} satisfying the norm-preserving property. In fact, it is sufficient to assume that $a_t(\cdot, \cdot)$ satisfies a G\r{a}rding inequality on $V$, see e.g. \cite{DautrayLions, Hain}. 
\end{remark}

\subsubsection{Discrete case}
\label{sec:UnitaryTimeEvolutionDiscrete}
As mentioned above, we would like also the discrete system to be norm-preserving. For time-stepping schemes, there are several methods ensuring norm-preservation. Note, that also certain space-time Petrov-Galerkin discretizations can be interpreted as a time-stepping method, so that one can investigate norm-preservation. In fact, in \cite{GabrieleDissertation, WDoerflerSchroedinger}, a Crank-Nicolson scheme in time is shown to be norm-preserving -- and in turn a Crank-Nicolson scheme can be shown to be a special case of a Petrov-Galerkin space-time discretization, \cite{r.andreev2013A,PateraUrbanImprovedErrorBound}. 
However, in a more general Petrov-Galerkin space-time discretization, we cannot hope to have an analogous time-stepping scheme, in particular since trial and test space are constructed in order to ensure stability (and not in order to arrive at a time-stepping scheme). 

\medskip
\noindent\textbf{Inf-sup-optimal discretization.}
Let $u_\delta^* \in \U_\delta$ be the unique solution of the ultra-weak Petrov-Galerkin formulation \eqref{eq:var-disc}, where $\U_\delta := S^*(\V_\delta)$. Let Assumption \ref{assumption:UnitaryTimeEvolution} hold, so that $u^* \in L_2(I;V) \cap H^1(I;V')$. Hence, it makes sense to choose $\U_\delta\subset L_2(I;V) \cap H^1(I;V')$, i.e., with extra regularity. This can easily be realized by choosing the test functions $\V_\delta$ with sufficient regularity. Then, we can perform integration by parts as above and obtain for all $v_{\delta}\in\V_\delta$ that
\begin{align*}
\mathrm{i} (u_{0,\delta}, v_{\delta}(0))_{H}
	&=(u_{\delta}^*, S^* v_{\delta})_\U  \\
	&\kern-50pt= \mathrm{i} \int_0^T\!\!{ \langle \partial_t u_{\delta}^*(t), v_{\delta}(t) \rangle_{V' \times V} \; dt} - \int_0^T{ a_t(u_{\delta}^*(t), v_{\delta}(t)) \; dt } + \mathrm{i} (u_{\delta}^*(0), v_{\delta}(0))_H, 
\end{align*}
where $u_{0,\delta} \in \U_{\delta}$ denotes a suitable approximation of the initial condition $u_0 \in V$. Hence, $u_{\delta}^*(0)=u_{0,\delta}$ and by taking real and imaginary parts,
\begin{align}
	\label{eq:UnitaryTimeEvolutionDiscrete1}
	\int_0^T\!\!{ \langle \partial_t u_{\delta}^*(t), v_{\delta}(t) \rangle_{V' \times V} \; dt} 
	= \int_0^T{ a_t(u_{\delta}^*(t), v_{\delta}(t)) \; dt }
	&=0.
\end{align}
However, we cannot deduce norm-preservation from \eqref{eq:UnitaryTimeEvolutionDiscrete1} since $\V_\delta \not \subset \U_\delta$, so that we cannot test \eqref{eq:UnitaryTimeEvolutionDiscrete1} with $u_\delta$ in place of $v_\delta$. 

Moreover, in the minimal regularity case (i.e., if Assumption \ref{assumption:UnitaryTimeEvolution} is not valid), the solution might not be in $C(\bar{I};H)$, so that $\| u_\delta^*(T)\|_H$ might not even be defined. This is only the case, when we choose a sufficiently smooth discretization $\V_\delta$ so that $\U_\delta\subset C(\bar{I};H)$, which can of course be realized by choosing e.g.\ splines of higher order. 

\begin{proposition}\label{prop:UnitaryTimeEvolDisc}
	Let Assumption \ref{assumption:UnitaryTimeEvolution} be satisfied and choose $\V_\delta$ in such a way that $\U_\delta=S^*(\V_\delta)\subset L_2(I;V) \cap H^1(I;V')\cap L_2(I;H^2(\Omega;\C))$. Then, the discrete solution is asymptotically norm-preserving, i.e.,
	\begin{align*}
		 \big\vert \Vert u^*_\delta(T) \Vert_H^2 - \Vert u^*_\delta(0) \Vert_H^2 \big\vert
		= \cO (\Dt +h).
	\end{align*} 
\end{proposition}
\begin{proof}
	By \eqref{eq:UnitaryTimeEvolutionDiscrete1}, we get for any $v_\delta\in\V_\delta$
	\begin{align*}
	\Vert u^*_\delta(T) \Vert_H^2 - \Vert u^*_\delta(0) \Vert_H^2
	&= \int_0^T{ \frac{d}{dt} \Vert u^*_\delta(t) \Vert_H^2 \; dt} 
	= 2\,\Re\!\int_0^T \langle \dot{u}^*_\delta(t), u^*_\delta(t)\rangle_{V'\times V}\, dt\\
	&= 2\,\Re\!\int_0^T \langle \dot{u}^*_\delta(t), u^*_\delta(t)-v_\delta(t)\rangle_{V'\times V}\, dt\\
	&\le 2\, \| \dot{u}^*_\delta\|_{L_2(I;V')}\, \| u^*_\delta - v_\delta\|_{L_2(I;V)}.
	\end{align*}
	Since this holds for any $v_\delta\in\V_\delta$, we can pass to the infimum and use
	\begin{align*}
		\inf_{v_\delta\in\V_\delta} \| u^*_\delta - v_\delta\|_{L_2(I;V)}
		= \cO (\Dt +h)
	\end{align*}
	by standard spline approximation results in \eqref{eq:SplineApprox}.\footnote{Using $r=0$, $\ell=1$ in time and $r=1$, $\ell=2$ in space.}  Finally, due to convergence as $\delta\to 0$ and the additional regularity, the term  $\| \dot{u}^*_\delta\|_{L_2(I;V')}$ is uniformly bounded.
\end{proof}

In case of higher regularity, the above result can be improved.

\begin{corollary}\label{cor:UnitaryTimeEvolDisc}
	If in addition to the assumptions of Proposition \ref{prop:UnitaryTimeEvolDisc} one has $u\in H^{k,m+1}(\Omega_T;\C)$, $1\le k\le p_{\text{time}}$, $1\le m\le p_{\text{space}}$ and $\U_\delta\subset H^{k,m+1}(\Omega_T;\C)$, then $\big\vert \Vert u^*_\delta(T) \Vert_H^2 - \Vert u^*_\delta(0) \Vert_H^2 \big\vert = \cO ((\Dt)^k +h^{m})$.
\end{corollary}
\begin{proof}
	The proof is the same as the proof of Proposition \ref{prop:UnitaryTimeEvolDisc} but using \eqref{eq:SplineApprox} with $r=k$, $\ell=1$ in time and $r=1$, $\ell=m+1$ in space.
\end{proof}

Hence, the quantity $d_\delta(T):=\big\vert \Vert u_{\delta}(T) \Vert_H - \Vert u_{\delta}(0) \Vert_H \big\vert$ can expected to be \enquote{small}. We will investigate this in our subsequent numerical experiments in \S\ref{Sec:4}.

\medskip
\noindent\textbf{Galerkin discretization.}
The above considerations show that the inf-sup-optimal discretization is asymptotically norm-preserving, but not exactly. If one insists of a norm-preserving discretization, we indicate a possible space-time discretization. The price to pay is --as maybe expected-- suboptimal inf-sup stability.

\begin{definition}
	For to spaces $\U_\delta\subset\U$, $\V_\delta\subset\V\subset\U$, we define the \emph{maximal relative distance} by
	\begin{align}
		\varepsilon_\delta &:= \sup_{v_\delta\in\V_\delta} \inf_{u_\delta\in\U_\delta} \frac{\| u_\delta - v_\delta\|_\U}{\| v_\delta\|_\V}.
	\end{align}
\end{definition}
Obviously, this quantity ensures that for all $v_\delta\in\V_\delta$ there exists some $u_\delta(v_\delta)\in\U_\delta$ such that $\| v_\delta - u_\delta(v_\delta)\|_\U\le \varepsilon_\delta \| v_\delta\|_\V$ and by triangle inequality
\begin{align}\label{eq:InequalityEpsilonProximal}
	(1-\varepsilon_\delta) \Vert v_{\delta} \Vert_\V 
	\leq \Vert u_\delta(v_\delta)\Vert_{\U} 
	\leq (1+\varepsilon_\delta) \Vert v_{\delta} \Vert_\V.
\end{align}

\begin{proposition}
	We have $\varepsilon_\delta = \cO ((\Dt)^k +h^{m})$, $1\le k\le p_{\text{time}}$, $1\le m\le p_{\text{space}}$, if $u^*, \U_\delta\subset H^{k,m}(\Omega_T)$.
\end{proposition}
\begin{proof}
	We use the spline approximation \eqref{eq:SplineApprox} for $r=0$, $\ell=k-1$ in time and $r=0$, $\ell=m-1$ to obtain that $\| u_\delta - v_\delta\|_\U \le C\, ((\Dt)^{k} +h^{m}) \| v_\delta\|_\V$, which yields the claim.\footnote{Note the higher order norm on the right-hand side.}
\end{proof}

\begin{proposition}
\label{prop:InfSupProximal}
Let $\U_{\delta}=S^*(\V_\delta)$, then the sesquilinear form $b(\cdot,\cdot)$, given in \eqref{eq:veryweak:1}, satisfies
\begin{align}
\label{eq:DiscreteInfSupProximal}
	\tilde{\beta}_{\delta}
	&:= \inf_{v_\delta\in\V_\delta} \sup_{w_\delta\in \V_\delta} 
		\frac{|b(v_\delta, w_\delta)|}{\| w_\delta\|_\V\, \| v_\delta\|_\V} 
		\geq 1-2\varepsilon_\delta.
\end{align}
In particular, for any $g \in \V'$ the Galerkin problem
\begin{align}
	\label{eq:UltraWeakFormulationProximal}
	\text{find } v_{\delta} \in \V_{\delta}: 
	\qquad b(v_{\delta}, w_{\delta}) = g(w_{\delta}) \qquad \text{for all } w_{\delta} \in \V_{\delta}
\end{align}
admits a unique solution provided that $\varepsilon_\delta<\frac12$.
\end{proposition}

\begin{proof}
Let $v_\delta\in\V_\delta$ and choose $u_\delta(v_\delta)\in\U_\delta$ satisfying \eqref{eq:InequalityEpsilonProximal}. Then, Proposition \ref{Lem:LBB} ensures that
\begin{align*}
	\sup_{w_\delta\in\V_\delta} 
		\frac{|b(u_\delta(v_\delta),w_\delta)|}{ \| w_\delta\|_\V} 
		= \| u_\delta(v_\delta)\|_\U\,.
\end{align*}
Denote by $Tv_\delta\in\V_\delta$ the argument of the previous supremum, i.e., $|b(u_\delta(v_\delta),Tv_\delta)|=\| u_\delta(v_\delta)\|_\U\, \| Tv_\delta\|_\V$. Then, we obtain by the boundedness of the bilinear form $b(\cdot,\cdot)$ with continuity constant being unity
\begin{align*}
	|b(v_\delta, Tv_\delta)|
	&= |b(u_\delta(v_\delta), Tv_\delta) + b(v_\delta-u_\delta(v_\delta), Tv_\delta)| \\
	&\ge |b(u_\delta(v_\delta), Tv_\delta)| - \| w_\delta-u_\delta(v_\delta)\|_\U\, \| Tv_\delta\|_\V \\
	&\ge \| u_\delta(v_\delta)\|_\U\, \| Tv_\delta\|_\V - \varepsilon_\delta \, \| v_\delta\|_\V \, \| Tv_\delta\|_\V \\
	&\ge (1-\varepsilon_\delta) \| v_\delta\|_\V \, \| Tv_\delta\|_\V - \varepsilon_\delta \, \| v_\delta\|_\V \, \| Tv_\delta\|_\V
	= (1-2\varepsilon_\delta) \| v_\delta\|_\V \, \| Tv_\delta\|_\V.
\end{align*}
Hence, 
\begin{align*}
	\sup_{w_\delta\in \V_\delta} 
		\frac{|b(v_\delta, w_\delta)|}{\| w_\delta\|_\V\, \| v_\delta\|_\V} 
	&\ge \frac{|b(v_\delta, Tv_\delta)|}{\| Tv_\delta\|_\V\, \| v_\delta\|_\V} 
	\ge 1-2\varepsilon_\delta,
\end{align*}
i.e., \eqref{eq:DiscreteInfSupProximal}. 
Moreover, continuity $|b(v_\delta, w_\delta)| \le \|v_\delta\|_\U\, \|w_\delta\|_\V \le \|v_\delta\|_\V\, \|w_\delta\|_\V$ is easily seen by Theorem \ref{Thm:stablecont}. Finally, for $v_\delta\not=0$, we have $b(v_\delta,v_\delta) = (v_\delta,S^*v_\delta)_\U\not=0$, so that the well-posedness of \eqref{eq:UltraWeakFormulationProximal} due to Theorem \ref{thm:Necas}.
\end{proof}

Now, we follow the same path as in the continuous case by specifying \eqref{eq:UltraWeakFormulationProximal} for the special case of a right-hand side, namely
\begin{align}
\label{eq:UnitaryTimeEvolutionDiscrete2}
	\text{find } v_{\delta} \in \V_{\delta}: \, 
		b(v_\delta,w_\delta) 
		= (v_{\delta}, S^* w_{\delta})_\U 
		= \mathrm{i} (u_{0,\delta}, w_{\delta}(0))_{H} 
			\quad \forall w_{\delta} \in \V_{\delta}.
\end{align}
As in the continuous case, integration by parts ensures that \eqref{eq:UnitaryTimeEvolutionDiscrete2} is equivalent to
\begin{align*}
	&\int_0^T\!\!\! \{ \mathrm{i} \langle \partial_t v_\delta(t), w_{\delta,1}(t) \rangle_{V' \times V} 
		\!-\!  a_t(v_\delta(t), w_{\delta,1}(t)) \} dt  
	+ \mathrm{i} \langle v_\delta(0),w_{\delta,2} \rangle_{V \times V'} 
	= \\
	&\qquad=\mathrm{i} \langle u_{\delta,0},w_{\delta,2} \rangle_{V \times V'}
\end{align*}%
for all $w_\delta=(w_{1,\delta},w_{2,\delta})\in\V_\delta\times V_h$.  Now, we can choose $w_{\delta}=(v_{\delta}-\mathrm{i} v_\delta(0))$ and taking imaginary parts yield $\|v_{\delta}(T)\|_H=\|u_{\delta}(0)\|_H$, i.e., the norm of the discrete initial data is preserved.

\begin{remark}
	From the above considerations, we can deduce that we cannot expect optimal inf-sup (i.e., error and residual coincide) and norm-preservation at the same time. Hence, one must decide which of the two properties is more important from an application point of view.
\end{remark}

\section{Ultra-Weak Numerical Solution}
\label{Sec:4}

In this section, we derive the algebraic system arising from the ultra-weak formulation and describe its solution. The potential is multiplicative (i.e., independent of the wave function), so that the arising algebraic system is a linear one.

\subsection{The linear system}
\label{SubSec:LinSystem}
To derive the stiffness matrix, we start by using arbitrary discrete spaces\footnote{The optimal ones will be considered $\psi_\mu=S^*(\varphi_\mu)$ below.} induced by $\{\psi_\mu := \sigma^\ell\otimes\xi_j :\, \mu=1,...,\cN_\delta\}$ for the trial and $\{  \varphi_\nu = \varrho^k\otimes\phi_i:\, \nu=1,...,\cN_\delta\}$ for the test space. Using the notation $[\bbS_\delta]_{\mu,\nu}=[\bbS_\delta]_{(\ell,j),(k,i)}$, we get
\allowdisplaybreaks
\begin{align}
		 [\bbS_\delta]_{(\ell,j),(k,i)} 
		&= b(\psi_\mu,\varphi_\nu) 
		= (\psi_\mu,S^*\varphi_\nu)_\U  
		 = ( \sigma^\ell\otimes\xi_j,
		 	(\mathrm{i}\partial_{t} + \textstyle{\frac12}\Delta_x - \cV) \varrho^k\otimes\phi_i
			)_\U \nonumber\\
		&= -\mathrm{i} (\sigma^\ell, \dot\varrho^k)_{L_2(I)}\,
				(\xi_j, \phi_i)_{L_2(\Omega)}
			+\textstyle{\frac12} (\sigma^\ell, \varrho^k)_{L_2(I)}\,
				(\xi_j, \Delta_x \phi_i)_{L_2(\Omega)}
		\label{Eq:Bdelta-general}\\
		&\qquad - ( \sigma^\ell\otimes\xi_j, \cV [\varrho^k\otimes\phi_i])_{L_2(\Omega_T)}.
			\nonumber
\end{align}
In the specific case $\psi_\mu=S^*(\varphi_\mu)$, we get the representation
\allowdisplaybreaks
\begin{align*}
	 [\bbS_\delta]_{(\ell,j),(k,i)} 
	&= b(\psi_\mu,\varphi_\nu) 
		= (\psi_\mu,S^*\varphi_\nu)_\U  
		= (S^*\varphi_\mu,S^*\varphi_\nu)_\U 
					\nonumber\\
		&= ( (\mathrm{i}\partial_{t} +\textstyle{\frac12}\Delta_x - \cV) 
				(\varrho^\ell\otimes\phi_j),
			(\mathrm{i}\partial_{t} +\textstyle{\frac12}\Delta_x - \cV)  
				(\varrho^k\otimes\phi_i))_{\U}.
\end{align*}
Now, we need to take the rules for complex integration into account and get
\allowdisplaybreaks
\begin{align}
	 [\bbS_\delta]_{(\ell,j),(k,i)} 
		&= (\dot\varrho^\ell,\dot\varrho^k)_{L_2(I)}\, (\phi_j, \phi_i)_{L_2(\Omega)}
			+ \textstyle{\frac14} (\varrho^\ell, \varrho^k)_{L_2(I)}\, (\Delta_x\phi_j, \Delta_x\phi_i)_{L_2(\Omega)}
				\nonumber\\
		&\quad - \textstyle{\frac12} ( \cV[\varrho^\ell\otimes\phi_j], \varrho^k\otimes\Delta_x\phi_i)_{L_2(\Omega_T)}
				- \textstyle{\frac12} ( \varrho^\ell\otimes\Delta_x\phi_j, \cV[\varrho^k\otimes\phi_i])_{L_2(\Omega_T)}
				\nonumber\\
		&\quad + ( \cV[\varrho^\ell\otimes\phi_j], \cV[\varrho^k\otimes\phi_i])_{L_2(\Omega_T)}
				\nonumber\\
		&\quad + \textstyle{\frac{\textrm{i}}2} (\dot\varrho^\ell, \varrho^k)_{I}\, (\phi_k, \Delta_x\phi_i)_{L_2(\Omega)}
			- \textstyle{\frac{\textrm{i}}2} (\varrho^\ell, \dot\varrho^k)_{I}\, (\Delta_x \phi_k, \phi_i)_{L_2(\Omega)}
				\nonumber\\
		&\quad +\mathrm{i} (\cV[\dot\varrho^\ell\otimes\phi_j], \varrho^k\otimes\phi_i)_{L_2(\Omega_T)}
			-\mathrm{i} (\varrho^\ell\otimes\phi_j, \cV[\dot\varrho^k\otimes\phi_i])_{L_2(\Omega_T)}
	\label{Eq:Bdelta-optimal}
\end{align}
so that 
\begin{align}
	\bbS_\delta
		&= \bA_\Dt \otimes \bM_h + \textstyle{\frac14} \bM_\Dt\otimes \bA_h 
			- \textstyle{\frac12} (\bbV_\delta +\bbV_\delta^T)  + \bbW_\delta 
				\nonumber\\
		&\quad +\mathrm{i}\Big[ \textstyle{\frac12}  \bN_\Dt\otimes \bN_h^T 
				- \textstyle{\frac12} \bN_\Dt^T \otimes \bN_h
				+ (\bbU_\delta - \bbU_\delta^T)\Big]
			=:  \bbA_\delta + \mathrm{i}\, \bbB_\delta,
		\label{eq:Sdelta}
\end{align}
where
\begin{align*}
	[\bA_\Dt]_{\ell,k} &:= (\dot\varrho^\ell,\dot\varrho^k)_{L_2(I)},
	&&&
	[\bA_h]_{j,i} &:= (\Delta_x\phi_j, \Delta_x\phi_i)_{L_2(\Omega)},\\
	[\bM_\Dt]_{\ell,k} &:= (\varrho^\ell,\varrho^k)_{L_2(I)},
	&&&
	[\bM_h]_{j,i} &:= (\phi_j, \phi_i)_{L_2(\Omega)},\\
	[\bN_\Dt]_{\ell,k} &:= (\dot\varrho^\ell,\varrho^k)_{L_2(I)},
	&&&
	[\bN_h]_{j,i} &:= (\Delta_x\phi_j, \phi_i)_{L_2(\Omega)},
\end{align*}
as well as
\begin{align*}
	[\bbU_\delta]_{(\ell,j),(k,i)} &:= (\cV[\dot\varrho^\ell\otimes\phi_j], \varrho^k\otimes\phi_i)_{L_2(\Omega_T)},  \\
	[\bbV_\delta]_{(\ell,j),(k,i)} &:= (\varrho^\ell\otimes\Delta_x\phi_j, \cV[\varrho^k\otimes\phi_i])_{L_2(\Omega_T)} ,\\
	[\bbW_\delta]_{(\ell,j),(k,i)} &:= ( \cV[\varrho^\ell\otimes\phi_j], \cV[\varrho^k\otimes\phi_i])_{L_2(\Omega_T)}.
\end{align*}
The latter three matrices $\bbU_\delta$, $\bbV_\delta$ and $\bbW_\delta$ involve the potential $\cV$ and are thus, in general, not separable. However, there are several methods to (at least) approximate $\cV(t,x)$ in terms of a sum of separable functions. Hence, let us for simplicity assume that $\cV$ is a tensorproduct, i.e., 
\begin{align}\label{eq:Vtenprod}
	\cV(t,x) = \Theta(t)\, \Xi(x),
		\qquad t\in I, x\in\Omega
\end{align}
for functions $\Theta\in L_\infty(I;\R)$ and $\Xi\in L_\infty(\Omega;\R)$. 
In that case, we can further detail the matrices $\bbU_\delta$, $\bbV_\delta$ and $\bbW_\delta$ as follows
\begin{align*}
	[\bbU_\delta]_{(\ell,j),(k,i)} 
	&= (\cV[\dot\varrho^\ell\otimes\phi_j], \varrho^k\otimes\phi_i)_{L_2(\Omega_T)}
		= (\Theta\, \dot\varrho^\ell, \varrho^k)_{L_2(I)}\, (\Xi\, \phi_j,\phi_i)_{L_2(\Omega)},\\
	[\bbV_\delta]_{(\ell,j),(k,i)} 
	&= (\varrho^\ell\otimes\Delta_x\phi_j, \cV[\varrho^k\otimes\phi_i])_{L_2(\Omega_T)}
		= (\varrho^\ell, \Theta\, \varrho^k)_{L_2(I)}\, (\Delta_x \phi_j, \Xi\, \phi_i)_{L_2(\Omega)},\\
	[\bbW_\delta]_{(\ell,j),(k,i)} 
	&= ( \cV[\varrho^\ell\otimes\phi_j], \cV[\varrho^k\otimes\phi_i])_{L_2(\Omega_T)}
		= (\Theta\,\varrho^\ell, \Theta\, \varrho^k)_{L_2(I)}\, (\Xi\,\phi_j, \Xi\, \phi_i)_{L_2(\Omega)}.
\end{align*}
Finally, let us now detail the right-hand side. Recall 
from \eqref{eq:veryweak:2}, that $g(v)= (f,v)_\U$. Hence,
\begin{align*}
	[\bg_\delta]_\nu 
		&= [\bg_\delta]_{(k,i)}
		= \langle g,\varrho^k\otimes \phi_i\rangle_{\V'\times\V}+ \mathrm{i} (\psi_0, \varrho^k(0)\otimes \phi_i)_{L_2(\Omega;\C)}\\
		&= \int_0^T \int_\Omega g(t,x)\, \varrho^k(t)\, \phi_i(x)\, dx\, dt
			+\mathrm{i} \varrho^k(0) \int_\Omega \psi_0(x)\, \phi_i(x)\, dx
\end{align*}
Using appropriate quadrature formulae results in a numerical approximation, which we will again denote by $\bg_\delta$. 
Then, solving the linear system $\bbS_\delta \bu_\delta=\bg_\delta$ yields the expansion coefficients of the desired approximation $u_\delta\in\U_\delta$ as follows: Let $\bu_\delta=(u_\mu)_{\mu=1,...,\cN_\delta}\in\C^{\cN_\delta}$, $\mu=(k,i)$, then we obtain
\begin{align*}
	u_\delta(t,x)
	&= \sum_{\mu=1}^{\cN_\delta}  u_\mu\, \psi_\mu(t,x)
	= \sum_{k=1}^{N_t}\sum_{i=1}^{N_h} u_{k,i}\, \sigma^k(x)\, \xi_i(x), 
\end{align*}
in the general case and for the special one, i.e., $\psi_\mu=S^*(\varphi_\mu)$ under the assumption \eqref{eq:Vtenprod}, we obtain
\begin{align}
	u_\delta(t,x)
	&= \sum_{\mu=1}^{\cN_\delta}  u_\mu\, \psi_\mu(t,x)
		\nonumber\\
	&= \sum_{k=1}^{N_t}\sum_{i=1}^{N_h} u_{k,i}\, 
		\big(\mathrm{i}\dot{\varrho}^k(t)\, \phi_i(x) 
				+\textstyle{\frac12} \varrho^k(t)\,  \Delta_x\phi_i(x)
				-\Theta(t)\varrho^k(t)\, \Xi(x)\phi_i(x)
				\big).
				\label{eq:numsolution}
\end{align}

\subsection{Numerical solution of the algebraic linear system}
\label{Subsec:AlgebraicLinearSystem}
For our numerical experiments we simply use  \MATLAB's \texttt{mldivide} (i.e., the backslash operator), which allows us to solve a system of linear equations with complex-valued entries. The development of more sophisticated solvers is devoted to future work, \cite{henning2021weak}. Recalling the form \eqref{eq:Sdelta} of the stiffness matrix and writing the vectors $\bu_\delta$ and $\bg_\delta$ of the desired solution and the right-hand side in the form $\bu_\delta = \bx_\delta + \mathrm{i}\, \by_\delta$ and $\bg_\delta = \bb_\delta + \mathrm{i}\, \bc_\delta$ with real-valued vectors of dimension $\cN_\delta$ each, we need to solve the block system
\begin{align*}
	\begin{pmatrix} \bbA_\delta & -\bbB_\delta \\ \bbB_\delta & \bbA_\delta \end{pmatrix}
	\begin{pmatrix} \bx_\delta \\ \by_\delta \end{pmatrix}
	=
	\begin{pmatrix} \bb_\delta \\ \bc_\delta \end{pmatrix}	
\end{align*}
to obtain real and imaginary part of the coefficient vector in  \eqref{eq:numsolution}.

\subsection{Conditioning}
\label{Subsec:Conditioning}

\begin{figure}[htb]
\centering
			\begin{tikzpicture}
				\begin{loglogaxis}[xlabel = {Dimension $\cN_\delta$}, ylabel = {$\kappa_2$},  legend cell align={left}, grid=both,width= 0.95\textwidth,height=5cm,xmin=20,xmax=10000]
					\addplot[mark=none, green, line width=1pt] table [x=NDOFS, y=COND] {results/results_smooth_levTime_minus1.txt};
					\addplot[mark=none, red, line width=1pt] table [x=NDOFS, y=COND] {results/results_smooth_levTime_plus0.txt};
					\addplot[mark=none, blue, line width=1pt] table [x=NDOFS, y=COND] {results/results_smooth_levTime_plus1.txt};
				\end{loglogaxis}
			\end{tikzpicture} 
			\caption{Spectral condition number $\kappa_2$ of system matrix $\bbS_{\delta}$ obtained by Schrödinger equation for a free particle, over the dimension $\cN_\delta$ (time and space) using different discretization strategies (time level vs.\ space level: green: $-1$, red: equal, blue: $+1$, see also \S\ref{Sec:5} below).}
			\label{Fig:ConditionNumber}
\end{figure}
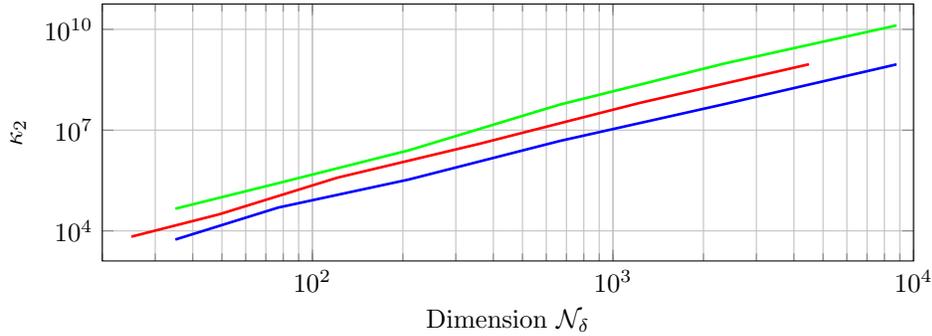

Since we use a direct method to solve the algebraic linear system, we did not use preconditioning for our numerical experiments. The reason is that we are mainly interested in the \emph{quantitative} performance of the proposed discretization and will devote the development of specific linear solvers to future work. To give an impression of that issue, we note, that following \cite{henning2021weak}, efficient iterative numerical solvers can be constructed by taking advantage of the tensor product nature of the linear system, preconditioning methods and a tailored Galerkin projection. In fact, for the spectral condition number $\kappa_2$ of the system matrix $\bbS_\delta$ obtained by the Schrödinger equation of a free particle, we observe a similar asymptotic behavior as for the system matrix of the Poisson problem, namely $\kappa_2 = \mathcal{O}(1/\delta^2)$, see Figure \ref{Fig:ConditionNumber}.

\section{Numerical Experiments}
\label{Sec:5}
In this section, we present some results of our numerical experiments concerning the proposed ultra-weak variational formulation for the Schrödinger equation.  We restrict ourselves to the Schrödinger equation for a free particle in a 1D potential well, i.e., for the unit interval $\Omega := (0,1)$, $I := (0,1)$, we consider the model problem \eqref{eq:SchroedingerEquation} with $\cV\equiv 0$ and $g\equiv 0$, which is a quite simple and well-understood problem in quantum mechanics, see e.g.\ \cite{CMollet} for some numerical investigations.

We choose uniform discretizations, i.e.,  dyadic uniform meshes in space and time of grid size $\Dt=2^{-j_{\text{time}}}$ and $h=2^{-j_{\text{space}}}$, i.e.: 
\begin{align*}
	\mathcal{T}_{\Delta t}^{\text{time}} = k \, \Delta t 
	\quad \text{and} \quad 
	\mathcal{T}_{h}^{\text{space}} = \ell \, h, 
	\qquad k=0,..., 2^{j_{\text{time}}}, \ 
	\ell = 0,..., 2^{j_{\text{space}}}.
\end{align*}
Here $j_{\text{time}}, j_{\text{space}} \in \mathbb{N}$ denote the refinement level of the temporal or spatial discretization, respectively. In order to demonstrate the stability of the method, we investigate several combinations of  $j_{\text{time}}$ and $j_{\text{space}}$, i.e., $j_{\text{time}} = j_{\text{space}} - 1$, $j_{\text{time}} = j_{\text{space}}$ and $j_{\text{time}} = j_{\text{space}} + 1$, see also Figure \ref{Fig:ConditionNumber} above.

On these grids, we build the Petrov-Galerkin discretization according to \S\ref{sec:StablePG}, namely $\mathbb{U}_{\delta} = S^*(\mathbb{V}_{\delta})$ as well as $\mathbb{V}_{\delta} = R_{\Delta t} \otimes Z_h$. For those cases, where we use continuous trial functions, $R_{\Delta t}$ is spanned by B-Splines of order 3 and $Z_h$ by B-Splines of order 4. We use the \textsc{Matlab} spline package from \cite{CMollet}.

\subsection{Uniform stability}

We first check the uniform stability, namely if we can observe that $\beta_{\delta} = 1$, independent of the discretization $\delta = (\Delta t, h)$. The results are shown in Table \ref{Tbl:InfSupConstants} conform this expectation. We do not need any CFL-condition fixing the relation of temporal and spatial discretization.

\newcolumntype{L}{>{$}l<{$}}
\newcolumntype{R}{>{$}r<{$}}
\newcolumntype{C}{>{$}c<{$}}

\begin{table}[htb]
{\scriptsize
	\begin{minipage}{0.32\textwidth}
		\begin{tabular}{|C|C|R|R|}
		\hline 
		j_{\text{space}} & j_{\text{time}} & \cN_\delta & \beta_{\delta} \\ 
		\hline 
		1 & 0 & 6 & 1.0 \\ 
		\hline 
		2 & 1 & 15 & 1.0 \\ 
		\hline 
		3 & 2 & 45 & 1.0 \\ 
		\hline 
		4 & 3 & 153 & 1.0 \\ 
		\hline 
		5 & 4 & 561 & 1.0 \\ 
		\hline 
		\end{tabular} 
	\end{minipage}
	\begin{minipage}{0.32\textwidth}
		\begin{tabular}{|C|C|R|R|}
		\hline 
		j_{\text{space}} & j_{\text{time}} & \cN_\delta & \beta_{\delta} \\ 
		\hline 
		1 & 1 & 9 & 1.0 \\ 
		\hline 
		2 & 2 & 25 & 1.0 \\ 
		\hline 
		3 & 3 & 81 & 1.0 \\ 
		\hline 
		4 & 4 & 289 & 1.0 \\ 
		\hline 
		5 & 5 & 1.089 & 1.0 \\ 
		\hline 
		\end{tabular} 
	\end{minipage}
	\begin{minipage}{0.32\textwidth}
		\begin{tabular}{|C|C|R|R|}
		\hline 
		j_{\text{space}} & j_{\text{time}} & \cN_\delta & \beta_{\delta} \\ 
		\hline 
		1 & 2 & 15 & 1.0 \\ 
		\hline 
		2 & 3 & 45 & 1.0 \\ 
		\hline 
		3 & 4 & 153 & 1.0 \\ 
		\hline 
		4 & 5 & 561 & 1.0 \\ 
		\hline 
		5 & 6 & 2.145 & 1.0 \\ 
		\hline 
		\end{tabular} 
	\end{minipage}
	\caption{Discrete inf-sup constant $\beta_{\delta}$ for different discretizations.}
	\label{Tbl:InfSupConstants}
}
\end{table}

\subsection{Approximation and norm preservation}
Next, we investigate the rate of approximation of the ultra-weak discretization as well as the quantity
\begin{align}
	\label{eq:ErrorTermNormPreservingProperty}
	d_\delta(T) := \big\vert \Vert u_{\delta}(T) \Vert_H - \Vert u_{\delta}(0) \Vert_H \big\vert,
\end{align} 
measuring the deviation from norm preservation, see Definition \ref{Def:NormPreserving} and Proposition \ref{prop:UnitaryTimeEvolDisc}. Both quantities may depend on the regularity of the solution, which in turn can be be controlled by varying the initial condition $u_0$. Hence, we investigate three cases, namely:
\begin{compactitem}
	\item[(a)] \textbf{Smooth case}: choose a smooth initial condition $u_0(x) = \sqrt{2} \sin{(\pi x)} \in C^{\infty}(\Omega) \cap C(\overline{\Omega})$,
	where $\Vert u_0 \Vert_H = 1$, which yields a closed form of the analytical solution $u(t,x) = u_0(x) \, \exp{(-\mathrm{i} {\textstyle{\frac{\pi^2} {2}}}t)}$, see also \cite{CMollet}.
	\item[(b)] \textbf{Sobolev case}: choose $u_0(x) = -3 \, \vert x-0.5 \vert + {\textstyle\frac32} \in H_0^1(\Omega)$, such that with $\Vert u_0 \Vert_{H} = 1$. Since an analytical solution is not available, we construct a reference solution by a sufficiently accurate time-stepping scheme, based on the standard semi-variational formulation. We used an implicit Euler scheme in time with $\Delta t = 2^{-14}$ and B-Splines of order $2$ and level $j_{\text{space}} = 11$ for spanning $V_h$ in space.
	\item[(c)] \textbf{Nonsmooth case}: The function $u_0(x) = 2 \cdot \mathds{1}_{[0.25,0.75]}(x) \in L_2(0,1)$ is discontinuous with $\Vert u_0 \Vert_{H} = 1$. Again, we compute a reference solution by time-stepping. Here, we use the implicit Euler scheme with $\Delta t = 2^{-14}$ and B-Splines of order $3$ and level $j_{\text{space}} = 11$ in space.
\end{compactitem}

\subsection{Approximation}
As we know that the regularity of the solution is determined by the smoothness of the initial condition, we expect the rate of approximation to depend on this. First, we note, that the \emph{asymptotic} rate of convergence does not depend on the relation of $j_{\text{time}}$ and $j_{\text{space}}$ as can be seen in Figure \ref{L2error-nonsmooth-disc}, where we show the error over the dimension of $\U_\delta$ for the three different choices. Hence, we restrict ourselves in the sequel to the case $j_{\text{time}}=j_{\text{space}}+1$.
\begin{figure}[htb]
	\begin{center}
	\begin{tikzpicture}
	\begin{loglogaxis}[xlabel = {$\cN_\delta$}, ylabel = {$L_2$ - Error},  legend cell align={left}, grid=both,width=0.95\textwidth,height=6cm,legend pos=south west,xmin=20,xmax=10000]
		\addplot[mark=none, green, line width=1pt] table [x=NDOFS, y=L2ERROR ] {results/results_L2_levTime_minus1.txt};
		\addlegendentry{$j_{\text{time}} = j_{\text{space}} - 1$}
		\addplot[mark=none, red, line width=1pt] table [x=NDOFS, y=L2ERROR,] {results/results_L2_levTime_plus0.txt};
		\addlegendentry{$j_{\text{time}} = j_{\text{space}}$}
		\addplot[mark=none, blue, line width=1pt] table [x=NDOFS, y=L2ERROR] {results/results_L2_levTime_plus1.txt};
		\addlegendentry{$j_{\text{time}} = j_{\text{space}} + 1$}
	\end{loglogaxis}
	\end{tikzpicture} 
	\caption{$L_2$-error over $\cN_\delta$ for different discretizations, nonsmooth case (c).\label{L2error-nonsmooth-disc}}
	\end{center}
\end{figure}
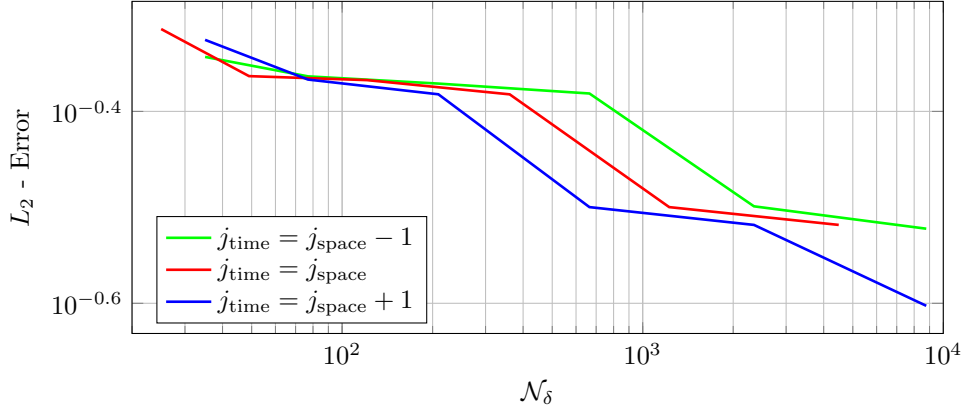

In Figure \ref{Error-comparision} (solid lines) we compare the rate of approximation for the three cases of initial condition. We clearly see the dependency and also observe that a sufficiently high regularity of the test (and hence also of the trial) functions yields the optimal rate (here linear since the trial functions are piecewise linears) of convergence. On the other hand, we also see that the ultra-weak discretization also converges in the minimal-regularity case (c), which is a challenge for time-stepping schemes w.r.t.\ both sufficient resolution and CFL-type stability.

\subsection{Norm preservation}
Finally, we investigate the (approximate) norm pre\-servation of our ultra-weak discretizations. First, similar to Figure \ref{L2error-nonsmooth-disc}, we report that we did not find any dependency on the relation of $j_{\text{space}}$ and $j_{\text{time}}$, so that we show results again only for $j_{\text{time}}=j_{\text{space}}+1$. We monitor the deviation from norm preservation over the dimension $\cN_\delta$ of trial and test space in Figure \ref{Error-comparision} (dashed lines). In order to see the relation to the discretization error, we plot both lines (error and norm preservation) in one figure. As we see, the slopes of both lines coincide for the smooth case and are slightly faster for the Sobolev and nonsmooth cases. In any case, the error w.r.t.\ norm preservation is in the range of the discretization accuracy, which seems quite reasonable.

\begin{figure}
	\begin{center}
	\begin{tikzpicture}[scale=0.99]
	\begin{loglogaxis}[xlabel = {$\cN_\delta$}, ylabel = {Errors},  legend cell align={left}, grid=both,legend pos=south west, legend style={font=\tiny},width=0.95\textwidth,height=7cm,xmin=30,xmax=10000]
		\addplot[mark=none, blue, line width=1pt] table [x=NDOFS, y=L2ERROR] {results/results_smooth_levTime_plus1.txt};
		\addlegendentry{(a) smooth, error}
		\addplot[dashed,mark=none, blue, line width=1pt] table [x=NDOFS, y=L2DIFF] {results/results_smooth_levTime_plus1.txt};
		\addlegendentry{(a) smooth, $d_\delta(T)$}
		\addplot[mark=none, red, line width=1pt] table [x=NDOFS, y=L2ERROR] {results/results_H1_levTime_plus1.txt};
		\addlegendentry{(b) Sobolev, error}
		\addplot[dashed,mark=none, red, line width=1pt] table [x=NDOFS, y=L2DIFF] {results/results_H1_levTime_plus1.txt};
		\addlegendentry{(b) Sobolev, $d_\delta(T)$}
		\addplot[mark=none, green, line width=1pt] table [x=NDOFS, y=L2ERROR] {results/results_L2_levTime_plus1.txt};
		\addlegendentry{(c) nonsmooth, error}
		\addplot[dashed,mark=none, green, line width=1pt] table [x=NDOFS, y=L2DIFF] {results/results_L2_levTime_plus1.txt};
		\addlegendentry{(c) nonsmooth, $d_\delta(T)$}
	\end{loglogaxis}
	\end{tikzpicture} 
	\caption{$\| u - u_{\delta}\|_{L_2(I;\Omega)}$ (solid, \enquote{error}) and $d_\delta(T)=\big\vert \Vert u_{\delta}(T) \Vert_H - \Vert u_{\delta}(0) \Vert_H \big\vert$ (dashed, \enquote{norm preservation})  over number of unknowns for different regularity of the solution, $j_{\text{time}}=j_{\text{space}}+1$.\label{Error-comparision}}
	\end{center}
\end{figure}
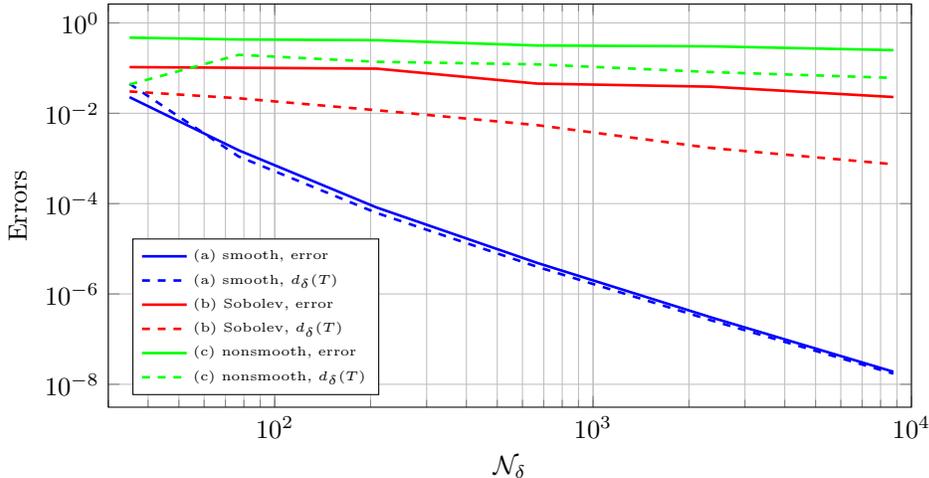

\section{Conclusions and Outlook}
\label{Sec:6}
In this paper we introduced a well-posed ultra-weak space-time variational formulation for the time-dependent linear Schrödinger equation and a corresponding optimally stable Petrov-Galerkin discretization. As a consequence, there is no need to satisfy other stability criteria like a CFL condition. On the other hand, however, this specific Petrov-Galerkin discretization only allows for \emph{asymptotic} norm-preservation in the same order as the discretization. As an alternative, we introduce a norm-preserving Galerkin scheme, loosing optimal stability though.

We view the results reported here as a starting point in several possible directions. Since we did not realize efficient numerical solvers adapted to the tensorproduct structure of the discretization, this seems to be a natural task following the ideas reported in \cite{henning2021weak}.

Next, as already said in the introduction, we followed the path of an ultra-weak variational form due to our interest in model reduction using the Reduced Basis Method, \cite{Haasdonk:RB,QuarteroniRB,RozzaRB}. In that framework, one could consider a time- and space-dependent potential $\mathcal{V}$, which one could be interested to be controlled in real-time, such that a desired state $u_D^*$ is achieved at final time $T$. 
This results in a bilinear optimal control problem, which can be solved e.g.\ by considering a parameterized optimization problem using the so called Control-To-State operator. An alternative could be the derivation of first order necessary conditions similar to \cite{BRU22}. In both cases, the potential $\mathcal{V}$ and the initial state $u_0$ can be seen as parameters in 
an infinite dimensional parameter space. The shown error/residual identity might be an important ingredient. There are many interesting open questions in that direction.

\bibliographystyle{ieeetr}
\bibliography{bibliography}
\end{document}